\documentclass{article}
\usepackage{graphicx}
\usepackage{epstopdf}
\usepackage{amsmath}
\usepackage{amssymb}
\usepackage{bm}
\usepackage{bbm}
\usepackage{enumerate}
\usepackage[T1]{fontenc}
\usepackage[latin1]{inputenc}
\usepackage{hyperref}
\usepackage{color}
\usepackage{curves}
\usepackage{tikz}
\usepackage[center]{caption}
\usepackage{ntheorem}
\usepackage{float}
\usepackage{mathrsfs}
\usepackage{wrapfig}

\title{The Deffuant model on $\Z$ with higher-dimensional
opinion spaces}
\author{Timo Hirscher\thanks{Research supported by a grant from the Swedish Research Council}
\\\normalsize Chalmers University of Technology}

\theoremstyle{break}
\newtheorem{theorem}{Theorem}[section]

\newtheorem{lemma}{Lemma}[section]
\newtheorem{proposition}{Proposition}[section]
\theorembodyfont{\upshape}
\newtheorem{definition}{Definition}
\newtheorem*{remark}{Remark}
\newtheorem{example}{Example}[section]
\makeatletter
\let\c@proposition\c@theorem
\let\c@lemma\c@theorem
\let\c@example\c@theorem
\makeatother

\newenvironment{proof}{\noindent{\sc Proof:}}{\vspace{-0.5cm}~\hfill $\square$\vspace{0.5cm}}

\newcommand\N{\mathbb{N}}
\newcommand\R{\mathbb{R}}
\newcommand\Z{\mathbb{Z}}
\newcommand\E{\mathbb{E}\,}
\newcommand\Prob{\mathbb{P}}
\renewcommand\epsilon{\varepsilon}
\renewcommand\phi{\varphi}
\DeclareMathOperator\supp{supp}
\DeclareMathOperator\conv{conv}

\newcommand\D{\mathscr{D}}
\newcommand{\n}[1]{\lVert {#1}\rVert_2}

\definecolor{darkblue}{rgb}{0,0,.5}
\hypersetup{colorlinks=true, breaklinks=true, linkcolor=blue, 
citecolor=darkblue, menucolor=blue, urlcolor=blue}

\begin{document}
\newpage
\maketitle
\begin{abstract}
When it comes to the mathematical modelling of social interaction patterns, a number of
different models have emerged and been studied over the last decade, in which individuals 
randomly interact on the basis of an underlying graph structure and share their opinions.
A prominent example of the so-called bounded confidence models is the one introduced by
Deffuant et al.: Two neighboring individuals will only interact if their opinions do not
differ by more than a given threshold $\theta$. We consider this model on the line graph
$\mathbb{Z}$ and extend the results that have been achieved for the model with real-valued
opinions by considering vector-valued opinions and general metrics measuring the distance
between two opinion values. As in the univariate case there turns out to exist a critical
value $\theta_\text{\upshape c}$ for $\theta$ at which a phase transition in the long-term
behavior takes place, but $\theta_\text{\upshape c}$ depends on the initial distribution in
a more intricate way than in the univariate case.
\end{abstract}



\section{Introduction}\label{intro}

Consider a simple graph $G=(V,E)$ and assume the vertex set $V$ to be either finite or countably
infinite with bounded maximal degree. The vertices are assumed to represent individuals and each of
them is assigned an opinion value.
The edges in $E$ -- being connections between individuals -- are understood to embody the possibility of mutual
influence. For that reason it is no restriction to focus on connected graphs, as the components could be treated
individually otherwise.
From different directions including social sciences, physics and mathematics, there has been raised
interest in various models for what is called {\itshape opinion dynamics} and deals with the evolution of
such a system under a given set of interaction rules. These models are qualitatively different but share similar ideas,
see \cite{Survey} for an extensive survey.\\[0.5em]
\noindent The {\em Deffuant model} (introduced by Deffuant et al.\ \cite{Model})
is one of those and features two parameters, the confidence bound $\theta>0$ and the convergence parameter 
$\mu\in(0,\tfrac 12]$, shaping the willingness to approach the other individual's opinion in a compromise.
There are two types of randomness in the model: One is the random {\em initial configuration}, meaning that at
time $t=0$ the vertices are assigned identically distributed opinions, the other are the {\em random encounters}
thereafter. Serving as a regime for the latter, all the edges in $E$ are assigned unit rate Poisson processes,
which are independent of one another and the initial configuration. Whenever a Poisson event occurs on an edge, the 
corresponding adjacent vertices interact in the manner described below. Just like in most of the analyses of this
model, we will consider i.i.d.\ initial opinion values, but comment on how the considerations can be generalized.

By $\eta_t(v)$ we denote the opinion value at vertex $v\in V$ at time $t\geq0$. The current value
will not change until at some future time $t$ a Poisson event occurs at one of the edges incident to $v$,
say $e=\langle u,v\rangle$, which then might cause an update. Let $\eta_{t-}(u):=\lim_{s\uparrow t}\eta_s(u)=a$ and
$\eta_{t-}(v):=\lim_{s\uparrow t}\eta_s(v)=b$ be the two opinion values of $u$ and $v$, just before this happens.

If these opinions lie at a distance less than the confidence bound $\theta$ from one another, they will
symmetrically take a step, whose size is scaled by $\mu$, towards a common compromise, if not they stay unchanged.
Although there is a section on vector-valued binary opinions in the original paper by Deffuant et al.\ \cite{Model},
using a different model, the Deffuant model with the interaction rule just described was originally only
defined for opinions being real-valued and the absolute value as notion of distance. In order to broaden the original
scope of this model to vector-valued opinions, the natural replacement for the absolute value is the Euclidean
distance $$d(x,y)=\n{ x-y}=\sqrt{(x-y)^2},\text{ for all }x,y\in\R^k.$$
Given this measure of distance, the rule for opinion updates in the Deffuant model reads as follows:

 \begin{equation*}
     \eta_t(u) = \left\{ \begin{array}{ll}
                     a+\mu(b-a) & \mbox{if $\n{ a-b}\leq\theta$,} \\
                     a & \mbox{otherwise}
                     \end{array} \right.
 \end{equation*}                     
and similarly \vspace{-0.75cm}\begin{align}\label{dynamics}\end{align}\vspace{-0.75cm}

 \begin{equation*}
     \eta_t(v) = \left\{ \begin{array}{ll}
                     b+\mu(a-b) & \mbox{if $\n{ a-b}\leq\theta$,} \\
                     b & \mbox{otherwise.}
                     \end{array} \right.
 \end{equation*}
Note that choosing $k=1$ gives back the original model. 

As the assumptions on the graph force $E$ to be countable, there will almost surely
be neither two Poisson events occurring simultaneously nor a limit point in time for the Poisson events on edges
incident to one fixed vertex.
Yet in addition to that there is a more subtle issue in how the simple pairwise interactions shape transitions
of the whole system in the infinite setting, putting it into question whether the whole process is well-defined
by the update rule (\ref{dynamics}).
For infinite graphs with bounded degree, however, this problem is settled by standard techniques
in the theory of interacting particle systems, see Thm.\ 3.9 on p.\ 27 in \cite{Liggett}.
\vspace*{1em}

\noindent One of the most natural questions in this context -- motivated by interpretations coming from social 
science -- seems to be, under what conditions the individual opinions will converge to a common consensus in the
long run and under what conditions they are going to split up into groups of individuals holding different opinions
instead. In this regard let us define the following types of scenarios for the asymptotic behavior of the Deffuant
model on a connected graph as time tends to infinity:

\begin{definition}\label{states}
\begin{enumerate}[(i)]
\item {\itshape No consensus}\\
There will be finally blocked edges, i.e.\ edges $e=\langle u,v\rangle$ s.t.
$$\n{\eta_t(u)-\eta_t(v)}>\theta,$$
for all times $t$ large enough. Hence the vertices fall into different opinion groups.
\item {\itshape Weak consensus}\\
Every pair of neighbors $\{u,v\}$ will finally concur, i.e.
$$\lim_{t\to\infty}\n{\eta_t(u)-\eta_t(v)}=0.$$
\item {\itshape Strong consensus}\\
The value at every vertex converges, as $t\to\infty$, to a common limit $l$, where
$$l=\begin{cases}\text{the average of the initial opinion values},&\text{if }G\text{ is finite}\\
                 \E\eta_0,&\text{if }G\text{ is infinite}\end{cases}$$
and $\mathcal{L}(\eta_0)$ denotes the distribution of the initial opinion values.\end{enumerate}
\end{definition}

\noindent 
The first analyses of the Deffuant model and similar opinion dynamics were strongly simulation-based
and thus confined to a finite number of agents. In \cite{sim1} for example, Fortunato simulated the long-term
behavior of the Deffuant model on four different kinds of finite graphs: Two deterministic examples -- the complete
graph and the square lattice -- as well as two random graphs -- those given by the Erd\H{o}s-R\'enyi model as well
as the Barab\'asi-Albert model. He found strong numerical evidence that, given initial opinions that are
independently and uniformly distributed on $[0,1]$, a confidence threshold $\theta$ less than $\tfrac12$ leads to a
fragmentation of opinions, $\theta>\tfrac12$ leads to a consensus -- irrespectively of the underlying
graph structures that were considered. Later, the simulation studies were extended to the generalization of the
Deffuant model to higher-dimensional opinion values, see for instance \cite{sim2}.

There are however crucial differences between the interactions on a finite compared to an infinite graph.
In the finite case, statements about consensus or fragmentation tend to be valid not with probability $1$ but at
best with a probability that is close to $1$: 
In the standard case of i.i.d.\ $\text{\upshape unif}([0,1])$ initial opinions for example,
any non-trivial confidence bound, i.e.\ $\theta\in(0,1)$, can lead to either consensus or fragmentation depending
on the initial values and the order of interactions. Furthermore, the fact that the dynamics (\ref{dynamics})
preserves the opinion average of two interacting agents implies that strong consensus follows from weak consensus
on a finite graph. This does not have to hold in an infinite setting.

The first major step in terms of a theoretical analysis of the model on an infinite graph was taken by Lanchier
\cite{Lanchier}, who treated the model on the line graph $\Z$ -- similarly with an i.i.d.\ $\text{\upshape unif}([0,1])$
configuration. His main result implies that there is a phase transition at $\theta=\tfrac12$ from a.s.\ no
consensus to a.s.\ weak consensus. These findings were reproven and slightly sharpened by Häggström \cite{ShareDrink} 
to the statement of Theorem \ref{on Z} below, using a non-random pairwise averaging procedure on
$\Z$ which he termed {\em Sharing a drink} (SAD) to get a workable representation of the opinion values at times
$t>0$.

Using his line of argument, the results were generalized to initial distributions other than $\text{\upshape unif}([0,1])$
by Häggström and Hirscher \cite{Deffuant} as well as Shang \cite{Shang}, independently. In \cite{Deffuant}, the
analysis of the Deffuant model was in addition to that extended to other infinite graphs, namely higher-dimensional
integer lattices $\Z^d$ and the infinite cluster of supercritical i.i.d.\ bond percolation on these lattices.

\vspace*{1em}
\noindent
In this paper we stay on the infinite line graph, that is the integer numbers $\Z$ with consecutive integers
forming an edge. The direction in which we want to broaden the analysis is -- as already indicated -- the
generalization of the Deffuant model on $\Z$ to vector-valued opinions. In Section 2, we give a brief
summary of the results for real-valued opinions derived in \cite{Deffuant}, together with the key ideas
and tools that were used there.

In Section 3 we establish corresponding results for the case of higher-dimen\-sional opinions sticking, as indicated
above, to the Euclidean norm as measure of distance between the opinions of interacting agents. Actually,
the main results (Theorem \ref{nogap} and \ref{gapsEucl}) in this section match the statement for real-valued
opinions (Theorem \ref{gen}) in the sense that the radius of the initial distribution as well as the largest gap
in its support -- the generalized definitions of which you will find in Definition \ref{radius} and \ref{gap} --
determine the critical value for $\theta$ at which there is a phase transition from a.s.\ no consensus to a.s.\ 
strong consensus. While the concept of a distribution's radius straightforwardly transfers to higher dimensions,
the one of a gap has to be properly redefined and investigated. Doing this, we can in fact characterize the support
of the opinion values at times $t>0$, see Proposition \ref{supp_t}. Even though we will throughout the paper
consider the initial opinions to be i.i.d.\ it is mentioned in the remark after Theorem \ref{gapsEucl}, how the
arguments can be extended to particular dependent initial configurations in the way it was done in \cite{Deffuant}.

Section 4 finally deals with the generalization of the Deffuant model to distance measures other
than the Euclidean, in both one and higher dimensions. We pin down properties a general metric $\rho$ (used to
determine whether two opinions are close enough to compromise or not) needs to have in order to allow for the
results from Section 3 to be preserved (see Theorem \ref{nogaprho} and \ref{gapsrho}). Examples are
given to illustrate the necessity of the requirements imposed on $\rho$.

\vspace*{1em}
\noindent
At this point it should be mentioned that the vectorial model that was already introduced in the original paper
by Deffuant et al.\ \cite{Model} and analyzed quite recently by Lanchier and Scarlatos \cite{Lanchier2} does not
fit the general framework of this paper. Unlike all opinion dynamics considered here, its update rule is different
from (\ref{dynamics}) and especially not average preserving, leading to substantial qualitative differences.

\section{Background on the univariate case}\label{1d}

\begin{theorem}[\bf Lanchier]\label{on Z}
    Consider the Deffuant model on the graph $(\Z,E)$, where $E=\{\langle v,v+1\rangle, v\in\Z\}$
    with i.i.d.\ {\upshape unif}$([0,1])$ initial configuration and fixed $\mu\in(0,\tfrac12]$.
    \begin{enumerate}[(i)]
    \item If $\theta>\tfrac12$, the model converges almost surely to strong consensus, i.e. with
    probability $1$ we have: $\lim_{t\to\infty}\eta_t(v)=\tfrac12$ for all $v\in\Z$.
    \item If $\theta<\tfrac12$ however, the integers a.s.\ split into (infinitely many) finite clusters
    of neighboring individuals asymptotically agreeing with one another, but no global consensus is approached.
    \end{enumerate}
\end{theorem}

\noindent
Accordingly, for independent initial opinions that are uniform on $[0,1]$, the critical value $\theta_\text{c}$
equals $\frac12$, with subcritical values of $\theta$ leading a.s.\ to no consensus and supercritical ones
a.s.\ to strong consensus. The case when the confidence bound actually takes on value $\theta_\text{c}$ is still
an open problem. The ideas Häggström \cite{ShareDrink} used to reprove the above result were adapted
to accommodate more general univariate initial distributions leading to a similar
statement for all such having a first moment $\E\eta_0\in\R\cup\{-\infty,+\infty\}$, see Thm.\ 2.2 in \cite{Deffuant},
which reads as follows: 

\begin{theorem}\label{gen}
Consider the Deffuant model on $\Z$ with real-valued i.i.d.\ initial opinions.
\begin{enumerate}[(a)]
   \item Suppose the initial opinion of all agents follows an arbitrary bounded distribution $\mathcal{L}(\eta_0)$
         with expected value $\E\eta_0$ and $[a,b]$ being the smallest closed interval containing its support.
         If $\E\eta_0$ does not lie in the support, let $I\subseteq[a,b]$ be the maximal, open interval such that
         $\E\eta_0$ lies in $I$ and $\Prob(\eta_0\in I)=0$. In this case let $h$ denote the length of $I$, otherwise 
         set $h=0$.
         
         Then the critical value for $\theta$, where a phase transition from a.s.\ no consensus to a.s.\ strong
         consensus takes place, becomes $\theta_\text{\upshape c}=\max\{\E\eta_0-a,b-\E\eta_0,h\}$.
         The limit value in the supercritical regime is $\E\eta_0$.
   \item Suppose the initial opinions' distribution is unbounded but its expected value exists, either in the 
         strong sense, i.e.\ $\E\eta_0\in\R$, or the weak sense, i.e.\ $\E\eta_0\in\{-\infty,+\infty\}$.
         Then the Deffuant model with arbitrary fixed parameter $\theta\in(0,\infty)$ will a.s.\ behave
         subcritically, meaning that no consensus will be approached in the long run.
\end{enumerate}
\end{theorem}
  \noindent
  The situation at criticality is unsolved with the exception of the case when the gap around the mean
  is larger than its distance to the extremes of the initial distribution's support. Given this condition, however,
  the following proposition (which is Prop.\ 2.4 in \cite{Deffuant}) settles the question about the long-term
  behavior for critical $\theta$:
  
\begin{proposition}\label{crit}
 Let the initial opinions be again i.i.d.\ with $[a,b]$ being the smallest closed interval containing
 the support of the marginal distribution,
 and the latter feature a gap $(\alpha,\beta)$ of width $\beta-\alpha>\max\{\E\eta_0-a,b-\E\eta_0\}$ around its
 expected value $\E\eta_0\in[a,b]$.\vspace*{0.5em}

\noindent At criticality, that is for $\theta=\theta_\text{\upshape c}
 =\max\{\E\eta_0-a,b-\E\eta_0,\beta-\alpha\}=\beta-\alpha$, we get
 the following: If both $\alpha$ and $\beta$ are atoms of the distribution $\mathcal{L}(\eta_0)$, i.e.\ 
 $\Prob(\eta_0=\alpha)>0$ and $\Prob(\eta_0=\beta)>0$, the system approaches a.s.\ strong consensus. However, it
 will a.s.\ lead to no consensus if either $\Prob(\eta_0=\alpha)=0$ or $\Prob(\eta_0=\beta)=0$.
\end{proposition} 
 
\noindent
  Since the same line of reasoning was used in both \cite{ShareDrink} and \cite{Deffuant} to derive the results
  we just stated, it is worth taking a closer look on the key concepts involved, especially as they will be the
  foundation for most of the conclusions drawn in the upcoming sections.

   The presumably most central among these is the idea of {\em flat points}. If $\E\eta_0\in\R$, a vertex $v\in\Z$
   is called {\em$\epsilon$-flat to the right} in the initial configuration $\{\eta_0(u)\}_{u\in\Z}$ if for all
   $n\geq0$:
   \begin{equation}\label{rflat}
     \frac{1}{n+1}\sum_{u=v}^{v+n}\eta_0(u)\in\left[\E\eta_0-\epsilon,\E\eta_0+\epsilon\right].
   \end{equation}
   It is called {\em$\epsilon$-flat to the left} if the above condition is met with the sum running
   from $v-n$ to $v$ instead. Finally, $v$ is called {\em two-sidedly $\epsilon$-flat} if for all $m,n\geq0$
   \begin{equation}\label{tflat}
     \frac{1}{m+n+1}\sum_{u=v-m}^{v+n}\eta_0(u)\in\left[\E\eta_0-\epsilon,\E\eta_0+\epsilon\right].
   \end{equation}
   However, in order to understand how vertices being one- or two-sidedly $\epsilon$-flat in the initial
   configuration play an important role in the further evolution of the configuration another concept is
   indispensable, namely the non-random pairwise averaging procedure Häggström \cite{ShareDrink} called
   {\em Sharing a drink} (SAD).
   
   Think of glasses being placed at all integers, the one at site $0$ being
   brimful, all others empty. Just as in the Deffuant model, neighbors interact and share, but this time without
   randomness and confidence bound. In other words, we start with the initial profile $\{\xi_0(v)\}_{v\in\Z}$,
   given by $\xi_0(0)=1$ and $\xi_0(v)=0$ for all $v\neq0$, and a finite sequence $(e_n)_{n=1}^N$ of edges
   along which updates of the form (\ref{dynamics}) are performed, i.e.\ for the profile $\{\xi_n(v)\}_{v\in\Z}$
   after step $n$ and $e_{n+1}=\langle u,u+1\rangle$ we get 
   $\{\xi_{n+1}(v)\}_{v\in\Z}$ by 
   \begin{equation}\label{transf}\begin{array}{rl}\xi_{n+1}(u)&\!\!\!=\,(1-\mu)\,\xi_{n}(u)+\mu\,\xi_{n}(u+1),\\
                 \xi_{n+1}(u+1)&\!\!\!=\,\mu\,\xi_{n}(u)+(1-\mu)\,\xi_{n}(u+1);\end{array}
   \end{equation}
   all other values stay unchanged.
   
   Elements of $[0,1]^\Z$ that can be obtained in such a way are called
   SAD-profiles. The crucial connection to the Deffuant model is that the opinion value $\eta_t(0)$ at any given
   time $t>0$ can be written as a weighted average of values at time $t=0$ with weights given by an SAD-profile
   (see La.\ 3.1 in \cite{ShareDrink}). The fact that all SAD-profiles share certain properties (the most important
   being unimodality) renders it possible to derive characteristics of the future evolution of the Deffuant dynamics
   given the initial configuration. For instance, the opinion value at a two-sidedly $\epsilon$-flat vertex in the
   initial configuration can never move further than $6\epsilon$ away from the mean (see La.\ 6.3 in
   \cite{ShareDrink}).
   
   These two vital ingredients -- flat points and SAD-profiles -- of the line of argument in \cite{ShareDrink}
   and Sect.\ 2 in \cite{Deffuant} can be adapted in order to analyze the Deffuant model with
   vector-valued opinions, as we will see in the following section.

\section{Deffuant model with multivariate opinions and the Euclidean norm as measure of distance}

Having characterized the long-term behavior of the Deffuant dynamics on $\Z$ starting from a general univariate
i.i.d.\ configuration, the next step of generalization with regard to the marginal initial distribution is,
as indicated in the introduction, to allow for vectors instead of numbers to represent the opinions. Like in
the univariate case, we want the initial opinions to be independent and identically distributed, just now with
some common distribution $\mathcal{L}(\eta_0)$ on $\R^k$. This will ensure ergodicity of the setting (with respect
to shifts) as before.

In this section we will consider $\R^k$ to be equipped with the Borel $\sigma$-algebra generated by the Euclidean
norm, denoted by $\mathcal{B}^k$.

\begin{definition}\label{radius}
If the distribution of $\eta_0$ has a finite expectation, define its {\em radius} by
$$R:=\inf\left\{r>0,\;\Prob\big(\eta_0\in B[\E\eta_0,r]\big)=1\right\},$$
where $B[y, r]:=\{x\in\R^k,\;\n{ x-y}\leq r\}$ denotes the closed Euclidean ball with radius $r$ around $y$.
Note that the radius of an unbounded distribution is infinite. 
\end{definition}

\noindent
The notion of $\epsilon$-flatness easily translates to the new setting by just replacing the intervals by balls:
If $\E\eta_0\in\R^k$, a vertex $v\in\Z$ is called $\epsilon$-flat to the right in the initial configuration
$\{\eta_0(u)\}_{u\in\Z}$ if for all $n\geq0$:
\begin{equation}\label{eucflat}
  \frac{1}{n+1}\sum_{u=v}^{v+n}\eta_0(u)\in B[\E\eta_0,\epsilon],
\end{equation}
similarly for $\epsilon$-flatness to the left and two-sided $\epsilon$-flatness -- compare with (\ref{rflat}) 
and (\ref{tflat}).

With these notions in hand we can state and prove a higher-dimensional analogue of Theorem \ref{gen}, valid for
initial distributions whose support does not feature a substantial gap around the mean. The proof of this result
will be a fairly straightforward adaptation of the methods for the univariate case indicated in Section \ref{1d}.
In contrast, the more general case treated in Theorem \ref{gapsEucl} requires invoking more intricate geometrical
considerations.

\begin{theorem}\label{nogap}
In the Deffuant model on $\Z$ with the underlying opinion space $(\R^k,\n{\,.\,})$ and an initial
opinion distribution $\mathcal{L}(\eta_0)$ we have the following limiting behavior:
\begin{enumerate}[(a)]
\item If $\mathcal{L}(\eta_0)$ has radius $R\in[0,\infty)$ and mass around its mean, i.e.
\begin{equation}\label{matm}
\Prob\big(\eta_0\in B[\E\eta_0,r]\big)>0 \text{ for all }r>0,
\end{equation}
the critical parameter is $\theta_\text{\upshape c}=R$, meaning that for $\theta<R$ we have a.s.\ no consensus
and for $\theta>R$ a.s.\ strong consensus.
\item Let $\eta_0=(\eta_0^{(1)},\dots,\eta_0^{(k)})$ be the random initial opinion vector. If at least one of the
coordinates $\eta_0^{(i)}$ has an unbounded marginal distribution, whose expected value exists (regardless of whether
finite, $+\infty$ or $-\infty$), then the limiting behavior will a.s.\ be no consensus, irrespectively of $\theta$.
\end{enumerate}
\end{theorem}

\begin{proof}
\begin{enumerate}[(a)]
\item To show the first part is just like in the univariate case (included in part (a) of Theorem 2.2) little
more than following the arguments in the last two sections of \cite{ShareDrink}: The central arguments go through 
even for vector-valued opinions as the crucial properties of the absolute value that were used are shared by its
replacement in higher dimensions, the Euclidean norm. Because of that, we only sketch the main line of reasoning
and refer to Sect.\ 6 in \cite{ShareDrink} and Sect.\ 2 in \cite{Deffuant} for a more thorough presentation of
the arguments.

First of all, the (multivariate) Strong Law of Large Numbers -- in the following abbreviated by SLLN -- tells us
that the averages in (\ref{eucflat}) for large $n$ are close to the mean in Euclidean distance. For $\epsilon>0$
fixed, choose $N\in\N$ such that the event
$$A:=\bigg\{\frac{1}{n+1}\sum_{u=1}^{n+1}\eta_0(u)\in B[\E\eta_0,\tfrac{\epsilon}{3}]\ \text{for all }n\geq N\bigg\}$$
has positive probability. Using (\ref{matm}) and the fact that the initial opinions are i.i.d., we can locally
modify the configuration to conclude that the event
$\{\eta_0(v)\in B[\E\eta_0,\tfrac{\epsilon}{3}] \text{ for }v=1,\dots,N+1\}\cap A$ has positive
probability, implying the $\epsilon$-flatness to the right of site $1$ -- just as it was done in La.\ 4.2 in
\cite{ShareDrink}.

For $\theta<R$, the probability of $\{\eta_0\notin B[\E\eta_0,\theta+\epsilon]\}$ is non-zero for $\epsilon$
small enough, hence a vertex can be at distance larger than $\theta$ from $B[\E\eta_0,\epsilon]$ initially.
Due to the independence of initial opinions, the event that site $-1$ is $\epsilon$-flat to the left, $1$ is
$\epsilon$-flat to the right and $\eta_0(0)\notin B[\E\eta_0,\theta+\epsilon]$ has positive probability.
Using the SAD representation, it follows -- mimicking Prop.\ 5.1 in \cite{ShareDrink} -- that given such an initial
configuration the opinion value at site $1$ will be a convex combination of averages in (\ref{eucflat}) for all times
$t>0$ and thus in $B[\E\eta_0,\epsilon]$, due to the convexity of Euclidean balls. The same holds for site $-1$ and the
half-line to the left. Consequently, the edges $\langle-1,0\rangle$ and $\langle0,1\rangle$ will stay blocked for ever.
Ergodicity of the initial opinion sequence ensures that with probability $1$ (infinitely many) vertices will
get isolated that way, which settles the subcritical case.

In the supercritical regime, i.e.\ $\theta>R$, we focus on two-sidedly $\epsilon$-flat vertices: If site
$0$ is $\epsilon$-flat to the left and $1$ is $\epsilon$-flat to the right, both are two-sidedly $\epsilon$-flat
-- using again the convexity of $B[\E\eta_0,\epsilon]$. By independence this event has positive probability, by
ergodicity we will a.s.\ have (infinitely many) two-sidedly $\epsilon$-flat vertices. Mimicking La.\ 6.3 in
\cite{ShareDrink} literally, we find that vertices which are two-sidedly $\epsilon$-flat in the initial configuration
will never move further than $6\epsilon$ away from the mean, irrespectively of future interactions. Choosing 
$\epsilon>0$ small, such that\ $7\epsilon<\theta-R$ say, will ensure that updates along edges incident
to two-sidedly $\epsilon$-flat vertices will never be prevented by the distance of opinions exceeding the confidence
bound.

The proof of Prop.\ 6.1 in \cite{ShareDrink}, which states that neighbors will either finally concur or the
edge between them be blocked for large $t$, can be adopted as well: Its central idea -- borrowed from physics -- that
every individual starts with an initial amount of energy that is then partly transferred partly lost in interactions
works regardless whether the opinions $\{\eta_t(v)\}_{v\in \Z}$ are shaped by numbers or vectors. Merely in the
current setting, the term $W_t(v)=(\eta_t(v))^2$, that defines the energy at vertex $v$ at time $t$, has to be read
as a dot product.
Again, if the opinions $\eta_{t}(u),\eta_{t}(v)$ of two neighbors are within the confidence
bound but $\n{\eta_{t}(u)-\eta_{t}(v)}\geq\delta$ for some fixed $\delta>0$, $W_t(u)+W_t(v)$ decreases by at least
$2\mu(1-\mu)\delta^2$ when they compromise. This can not happen infinitely often with positive probability
as the expected energy at time $t=0$ is $\E W_0(v)=\E(\eta_0^{\,2})<\infty$ and the expectation of $W_t(v)$ is
both non-increasing with $t$ and non-negative. For details see Prop.\ 6.1 and La.\ 6.2 in \cite{ShareDrink}.

Following from the considerations above, two-sidedly $\epsilon$-flat vertices and their neighbors therefore have
to finally concur with probability $1$, forcing the opinion values of the neighbors to eventually lie at a distance
strictly less than $7\epsilon$ from the mean as well. By our choice of $\epsilon$, this conclusion propagates
inductively showing that the limiting behavior will a.s.\ be strong consensus, if we let $\epsilon$ tend to $0$.

\item In order to prove the second claim, we use part (b) of Theorem \ref{gen}, focussing
on the $i$th coordinate only. Fix $\theta\in(0,\infty)$. Since $$|x_i-y_i|\leq\n{x-y} 
\text{ for all vectors }x,y\in\R^k \text{ and }i\in\{1,\dots,k\},$$
a distance of more than $\theta$ in the $i$th coordinate of the opinion vectors for two neighbors $u,v$ implies that
the edge between them is blocked. The arguments used for unbounded
distributions in Theorem \ref{gen} (see Thm.\ 2.2 in \cite{Deffuant}) show that under the given conditions, there are
a.s.\ vertices that differ more than $\theta$ from both their neighbors in the $i$th coordinate (with respect to
the absolut value) in the initial configuration and this will not change no matter whom their neighbors will
compromise with. Consequently, the corresponding opinion vectors will always be at Euclidean distance more than $\theta$.
\end{enumerate}
\end{proof}\vspace*{-1em}

\begin{remark}
  Pretty much as in the univariate setting, the case where all unbounded coordinates of $\eta_0$ do not have
  an expected value (neither finite nor $+\infty$ nor $-\infty$) remains unsolved by Theorem \ref{nogap}.
\end{remark}

\noindent
When it comes to bounded initial distributions which do have a large gap around the mean, the picture in higher
dimensions drastically changes -- something that 
\par\begingroup \rightskip15em\noindent
will require several preliminary results before we are ready to state and prove this section's main result,
Theorem \ref{gapsEucl}. The major difference to the univariate case
is that with higher-dimensional opinions the update along some edge $\langle u,v\rangle$
can actually lead to a situation, where both $u$ and $v$ come closer to the opinion corresponding to a third
vertex $w$, which lies within the confidence bound of neither $\eta(u)$ nor $\eta(v)$, see the picture on the right.
\par\endgroup

\vspace{-5.0cm}
   \begin{figure}[H]
     \hspace{6.8cm} \includegraphics[scale=1.1]{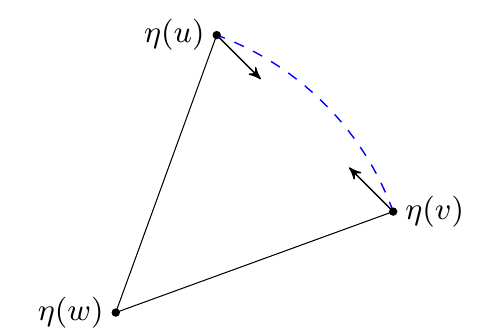}
   \end{figure}
 \vspace{0.38cm}

In the case of real-valued opinions this is impossible, because in that setting an update along
$\langle u,v\rangle$ always increases $\min\{|\eta(u)-\eta(w)|, |\eta(v)-\eta(w)|\}$, if $\eta(w)$ does not lie in
between $\eta(u)$ and $\eta(v)$.

To illustrate how this changes the conditions, let us consider the initial distributions $\text{\upshape unif}(S^{k-1})$,
where $S^{k-1}$ denotes the Euclidean unit sphere in $\R^k$.
For $k=1$ this is just $\text{\upshape unif}(\{-1,1\})$, which by Theorem \ref{gen} has the trivial critical value
$\theta_c=2$. For $k\geq2$ however, the fact that opinions close to each other can compromise in order to form
a central opinion will bring $\theta_c$ down to the radius 1 of the distribution as we will see in the sequel.

The statement of the main result in this section, Theorem \ref{gapsEucl}, resembles very much the one of
Theorem \ref{gen} (a), only the notion of a gap in the initial distribution has to be reinterpreted in the
higher-dimensional setting, making the proof of this generalized result rather technical. However, while establishing
auxiliary results, we will gain additional information about the set of opinion values that can occur in the
Deffuant model at times $t>0$ depending on the initial distribution and the confidence bound. When it comes to the
initial distribution $\mathcal{L}(\eta_0)$, the most important features besides its expected value are its support
and the corresponding radius.

\begin{definition}
Consider an $\R^k$-valued random variable $\zeta$. Its {\itshape support}
is the following subset of $\R^k$, which is closed with respect to the Euclidean metric:
$$\supp(\zeta):=\left\{x\in\R^k,\;\Prob\big(\zeta\in B[x,r]\big)>0\ \text{for all }r>0\right\}.$$
\end{definition}

\noindent
Observe that this definition corresponds to the standard notion of {\em spectrum of a measure} (see for example
Thm.\ 2.1 and Def.\ 2.1 in \cite{Partha}) -- applied to the distribution of a random variable.

If the initial distribution has a finite expectation, the radius can also be written as
$$R=\sup\left\{\n{\E\eta_0-x},\; x\in\supp(\eta_0)\right\},$$
as the following proposition shows.

\begin{proposition}\label{Radius}
If $\E\eta_0\in\R^k$, we have
\begin{equation}\label{rad} \inf\left\{r>0,\;\Prob\big(\eta_0\in B[\E\eta_0,r]\big)=1\right\}=
\sup\left\{\n{\E\eta_0-x},\; x\in\supp(\eta_0)\right\}.\end{equation}
\end{proposition}

\begin{proof}
First, consider a set $A$ which is compact in $(\R^k,\n{\,.\,})$ and a subset of the complement
of $\supp(\eta_0)$. The claim is that these properties of $A$ imply $\Prob(\eta_0\in A)=0$. Indeed, for every
$x\in A\subseteq (\supp(\eta_0))^\text{\upshape c}$ there exists $r_x>0$ s.t.\ $\Prob\big(\eta_0\in B[x,r_x]\big)=0$.
Let $B(y,r)$ denote the open Euclidean ball with radius $r$ around $y$, then $\{B(x,r_x),\;x\in A\}$ is an
open cover of $A$, which by compactness has a finite subcover $\{B(x_i,r_{x_i}),\;1\leq i\leq n\}$. Consequently
$$\Prob(\eta_0\in A)\leq\Prob\Big(\eta_0\in \bigcup_{i=1}^n B[x_i,r_{x_i}]\Big)=0.$$
If $r$ is greater than the supremum in (\ref{rad}) it follows that $\supp(\eta_0)\subseteq B(\E\eta_0,r)$.
Since 
$$\big(B(\E\eta_0,r)\big)^\text{\upshape c}=\Bigg(B[\E\eta_0,r+1]\setminus B(\E\eta_0,r)\Bigg)\cup
\Bigg(\bigcup_{q\in\mathbb{Q}^k\setminus B[\E\eta_0,r+1]}B[q,1]\Bigg)$$
and the right-hand side is a countable union of nullsets with respect to $\mathcal{L}(\eta_0)$, we get
$\Prob\big(\eta_0\in B[\E\eta_0,r]\big)=1$, which means that $r$ is greater or equal to the infimum in (\ref{rad}).

On the other hand, if $r$ is less than the supremum, there exists a point $x\in\supp(\eta_0)\setminus B[\E\eta_0,r]$,
which consequently has a positive distance $\delta$ to the closed ball $B[\E\eta_0,r]$. This gives
$$\Prob\big(\eta_0\in B[\E\eta_0,r]\big)\leq1-\Prob\big(\eta_0\in B[x,\tfrac{\delta}{2}]\big)<1.$$
In other words, $r$ does not appear in the set the infimum is taken over. Putting both arguments together
proves (\ref{rad}).
\end{proof}

\begin{definition}\label{Dtheta}
\begin{enumerate}[(i)]
\item For a finite graph $G=(V,E)$ and an edge $e=\langle u,v\rangle\in E$ let the update described in
(\ref{dynamics}), considered as a deterministic map on the set of $\R^k$-valued profiles, be denoted by $T_e^\theta$.
So if $T_e^\theta$ is applied to $\xi=\{\xi(v)\}_{v\in V}$ it just means that all values stay unchanged with the
only exception of
\begin{equation}\label{T_e}\left(\begin{array}{c}T_e^\theta\xi(u)\\T_e^\theta\xi(v)\end{array}\right)
=\left(\begin{array}{c}(1-\mu)\,\xi(u)+\mu\,\xi(v)\\\mu\,\xi(u)+(1-\mu)\,\xi(v)\end{array}\right)\quad\text{if } \n{\xi(u)-\xi(v)}\leq\theta.\end{equation}
\item Consider a finite section $\{1,\dots,n\}$ of the line graph, a finite
sequence $(e_i)_{i=1}^N$ of edges $e_i\in\{\langle1,2\rangle,\dots,\langle n-1,n\rangle\}$ and some values
$x_1,\dots,x_n$ in $\supp(\eta_0)$. Such a triple will from now on be called a {\itshape finite configuration}.\\
To {\em update the configuration} (with respect to $\theta$) will mean that we take
$x_1,\dots,x_n$ as initial opinions, i.e.\ we set $\eta_0(v)=x_v$ for all $v\in\{1,\dots,n\}$, and then apply
$T_{e_N}^\theta\circ T_{e_{N-1}}^\theta\circ\ldots\circ T_{e_1}^\theta$ to $\{\eta_0(v)\}_{v\in \{1,\dots,n\}}$.

Slightly abusing the notation, let the outcome, i.e.\ the final opinion values
$\{T_{e_N}^\theta\circ\ldots\circ T_{e_1}^\theta\,\eta_0(v)\}_{v\in \{1,\dots,n\}}$, be denoted by
$\{\eta_N(1),\dots,\eta_N(n)\}$.

\item Let $\nu$ denote the initial distribution $\mathcal{L}(\eta_0)$.
For $\theta>0$, let $\D_\theta(\nu)$ denote the set of vectors in $\R^k$ which the opinion values of
finite configurations can collectively approach, if updated according to confidence bound $\theta$.
More precisely, $x\in\D_\theta(\nu)$ if and only if for all $r>0$, there exist some $n\in\N=\{1,2,\dots\}$,
$x_1,\dots,x_n\in\supp(\eta_0)$ and $(e_i)_{i=1}^N$ as above, such that updating the configuration with respect to
$\theta$ yields $\eta_N(v)\in B[x,r]$ for all $v\in\{1,\dots,n\}$.
\end{enumerate}
\end{definition}

\noindent 
It is worth emphasizing that finite configurations are supposed to mimick the dynamics of the Deffuant model,
interpreting $(e_i)_{i=1}^N$ as the locations of the first $N$ Poisson events
on the edges $\langle0,1\rangle,\langle1,2\rangle,\dots,\langle n-1,n\rangle,\langle n,n+1\rangle$ in (strict)
chronological order. In this respect, considering $\theta$, we can choose the sequence $(e_i)_{i=1}^N$ such that
only Poisson events causing an actual update are considered by simply eliminating all events on edges where the
opinions of the two vertices are more than $\theta$ apart.

Note that according to the definition, $\D_\theta(\nu)$ depends on $\supp(\eta_0)$ and $\theta$,
as well as $\mu$, the latter being less obvious. See Example \ref{mu} below for an instance where $\mu$ actually
makes a difference. Let us now turn to various properties of the set $\D_\theta(\nu)$.

\begin{lemma}\label{properties}
Fix the distribution $\nu$ of $\eta_0$ and let $\D_\theta(\nu)$ and $R$ be defined as above.
\begin{enumerate}[(a)]
  \item $\D_\theta(\nu)$ is closed and increases with $\theta$.
  \item $\supp(\eta_0)\subseteq \D_\theta(\nu)\subseteq \overline{\conv(\supp(\eta_0))}\subseteq B[\E\eta_0,R]$ for all
        $\theta>0$, where $\conv(A)$ denotes the convex hull, $\overline{A}$ the closure of a set $A$.
\end{enumerate}
\end{lemma}

\begin{proof}
\begin{enumerate}[(a)]
  \item The first claim follows directly from the definition: For a sequence $(x_n)_{n\in\N}$ in $\D_\theta(\nu)$
    such that $\n{ x-x_n}\to 0$ and every $r>0$, there exists some $x_n \in B[x,\tfrac{r}{2}]$. Due to $x_n\in
    \D_\theta(\nu)$, there exists a finite configuration with all final opinion values in $B[x_n,\tfrac{r}{2}]$.
    But since $B[x_n,\tfrac{r}{2}]\subseteq B[x,r]$, this implies $x\in\D_\theta(\nu)$.
    
    As for the second claim, since we are free to choose the edge sequence in finite configurations, it is obvious
    that making $\theta$ larger only allows for more options when we are to come up with a setting that brings the
    opinion values collectively inside $B[x,r]$ for some given $x\in\R^k$ and $r>0$.
  \item The first inclusion is trivial, as for $x\in\supp(\eta_0)$ the finite configuration with $n=1,\ x_1=x$ will do.
    The second inclusion is due to the fact that every update of opinions is a convex combination, see (\ref{T_e}).
    Consequently, all final opinion values of finite configurations lie within $\conv(\supp(\eta_0))$. The last
    inclusion, which is meaningful only for $R<\infty$, follows from Proposition \ref{Radius} and the fact that
    $B[\E\eta_0,R]$ is both convex and closed.  
\end{enumerate}
\end{proof}

\noindent
It should be mentioned that an easy corollary to Carath\'eodory's Theorem on the convex hull states that the
convex hull of a compact set in $\R^k$ is compact as well. If $\eta_0$ has a bounded support, this implies that
the convex hull of $\supp(\eta_0)$ is actually closed, i.e.\ $\overline{\conv(\supp(\eta_0))}=\conv(\supp(\eta_0))$.

\begin{example}\label{jump}
To get familiar with the idea behind $\D_\theta(\nu)$, let us consider the discrete real-valued initial distribution
given by $\Prob(\eta_0=\tfrac1n)=\tfrac{1}{2^n}, n\in\N$. It is not hard to see that this implies
$\supp(\eta_0)=\{\tfrac1n,\; n\in\N\}\cup\{0\}$. Having the Taylor expansion of the logarithm in mind we find
$$\E\eta_0=\sum_{n=1}^\infty \frac{1}{n\,2^n}=-\left(-\sum_{n=1}^\infty \frac{(\tfrac12)^n}{n}\right)
  =-\ln(1-\tfrac12)=\ln(2).$$
By Theorem \ref{gen} we get $\theta_c=R=\ln(2)$, since $\Prob(\eta_0\in[0,1])=1$ and the largest gap in between the point
masses is $\tfrac12$.

For two point masses situated at $x$ and $y$ at distance $0<\n{ x-y}\leq\theta$, all convex combinations of
$x,y$ are in $\D_\theta(\nu)$: For $\alpha\in[0,1]$ and $r>0$, take $m,n\in\N$ s.t.\ 
$$\left|\frac{m}{m+n}-\alpha\right|\leq\frac{r}{4\,\max\{\n{x},\n{y}\}}.$$
Let us set up a finite configuration with $m+n$ vertices, $x_1=\ldots=x_m=x$ and $x_{m+1}=\ldots=x_{m+n}=y$ as well as enough
Poisson events on every edge (in an appropriate order) such that -- having updated the configuration according to the edge sequence -- the outcome $\eta_N(v)$ will
be at distance less than $\tfrac{r}{2}$ from the average $\tfrac{m}{m+n}\,x+\tfrac{n}{m+n}\,y$ for all
$v\in\{1,\dots,m+n\}$. Since all the opinion values lie in an interval of length at most $\theta$ in the beginning and hence
always will, we could choose the edge sequence by always taking the edge with largest current discrepancy next, to see
that a finite sequence with the claimed property exists. This will ensure
\begin{align*}
\n{\eta_N(v)-(\alpha x+(1-\alpha)y)}&\leq\tfrac{r}{2}+\n{(\tfrac{m}{m+n}\,x+\tfrac{n}{m+n}\,y)-(\alpha x+(1-\alpha)y)}\\
&\leq\tfrac{r}{2}+|\tfrac{m}{m+n}-\alpha|\cdot\n{x}+|\alpha-\tfrac{m}{m+n}|\cdot\n{y}\\
&\leq r,
\end{align*}
hence $\alpha x+(1-\alpha)y\in\D_\theta(\nu)$. This observation together with the fact that gaps of width
larger than $\theta$ can not be bridged leads to
$$\D_\theta(\nu)=[0,\tfrac{1}{n_\theta}]\cup\{\tfrac1n,\; n< n_\theta\},$$
where $n_\theta:=\max\{n\in\N,\;\tfrac{1}{n-1}-\tfrac{1}{n}>\theta\}$.
\end{example}

\begin{lemma}\label{circle}
\begin{enumerate}[(a)]
  \item For all $x\in\R^k$ and $0\leq \delta<\tfrac{\theta}{2}$, the set $\D_\theta(\nu)\cap B[x,\delta]$ is convex.
  \item If $R<\infty$, then $\D_{2R}(\nu)=\overline{\conv(\supp(\eta_0))}=\conv(\supp(\eta_0))$.
  \item The connected components of $\D_\theta(\nu)$ are convex and at distance at least $\theta$ from one another.
        If $\D_\theta(\nu)$ is connected, then $\D_\theta(\nu)=\overline{\conv(\supp(\eta_0))}$.
  \item If $R<\infty$ and $\nu$ has mass around its mean, i.e.\ condition (\ref{matm}) holds, then
        $\D_{\theta}(\nu)=\conv(\supp(\eta_0))$ already for $\theta>R$.
  \item For $R<\infty$, the set-valued mapping 
        $$\begin{cases}(0,\infty)\to\mathcal{B}^k\\
        \vartheta\mapsto \D_\vartheta(\nu)\end{cases}$$ is piecewise constant with only
        finitely many jumps on $[\delta,\infty)$ for all $\delta>0$.
  \item If $\D_\theta(\nu)$ is connected and $\E\eta_0$ finite, then $\E\eta_0\in\D_\theta(\nu)$
\end{enumerate}
\end{lemma}

\begin{proof}
\begin{enumerate}[(a)]
  \item The proof of the first part of this lemma follows the idea of the above example. Let $y,z\in\D_\theta(\nu)$
        and their distance be $0<\n{ y-z}\leq 2\delta<\theta$.
        Let $\epsilon=\theta-2\delta>0$. For any $\epsilon\geq r>0$, there exist finite configurations $\chi_1$ and
        $\chi_2$ with final values in $B[y,\tfrac{r}{4}]$ and $B[z,\tfrac{r}{4}]$ respectively.
        For $\alpha\in[0,1]$ choose again $m,n\in\N$ s.t.\
        $$\left|\frac{m}{m+n}-\alpha\right|\leq\frac{r}{4\,\max\{\n{y},\n{z}\}}.$$
        We define a new finite configuration by putting $m$ copies of $\chi_1$ and $n$ copies of $\chi_2$ next
        to each other: Their finite sections of the line graph (together with the assigned initial values) will be
        concatenated blockwise -- the order among the blocks being irrelevant -- by adding an edge between two
        consecutive blocks in order to form the underlying line graph of a larger finite configuration. To get an
        edge sequence for the whole configuration we will simply string together the edge sequences of the
        individual copies, again in a blockwise manner and arbitrary order.
        
        Updating according to the edge sequence will then bring all the opinion values within distance $\theta$
        of one another. Therefore, we can bring the final outcomes arbitrarily close, say at distance at most 
        $\tfrac{r}{4}$, to the average of the initial values, let's denote it by $\overline{x}$, by just adding a
        large enough (but finite) number of Poisson events on each edge (appropriately ordered as before). From the
        properties of the chosen building blocks, $\chi_1$ and $\chi_2$, it readily follows
        that the initial average is at distance at most $\tfrac{r}{4}$ from $\tfrac{m}{m+n}\,y+\tfrac{n}{m+n}\,z$.
        This entails for every vertex $v$ of the finite configuration
        \begin{align*}
         \n{\eta_N(v)-(\alpha y+(1-\alpha)z)}&\leq\tfrac{r}{4}+\n{\overline{x}-(\alpha y+(1-\alpha)z)}\\
         &\hspace{-2cm}\leq\tfrac{r}{4}+\tfrac{r}{4}+\n{(\tfrac{m}{m+n}\,y+\tfrac{n}{m+n}\,z)-(\alpha y+(1-\alpha)z)}\\
         &\hspace{-2cm}\leq\tfrac{r}{2}+|\tfrac{m}{m+n}-\alpha|\cdot\n{y}+|\alpha-\tfrac{m}{m+n}|\cdot\n{z}\\
         &\hspace{-2cm}\leq r,
        \end{align*}
        which shows $\alpha y+(1-\alpha)z\in\D_\theta(\nu)$.
  \item By Lemma \ref{properties} it is enough to show $\D_{2R}(\nu)\supseteq\conv(\supp(\eta_0))$. Thus, letting 
        $x,y\in\supp(\eta_0)\subseteq B[\E\eta_0,R]$, we have to show that $\conv(\{x,y\})\subseteq\D_{2R}(\nu)$.
        But since $\n{x-y}$ can be at most $2R$, this is done as described in Example \ref{jump}, just the line segment
        $\conv(\{x,y\})$ plays now the role of the interval considered there.
  \item First of all, the connected components of $\D_\theta(\nu)$ are actually path-connected and moreover the pathes
        can be chosen to be polygonal chains: Assume that a connected component $C$ contains more than one path-connected
        component. Fix one such, say $C_1$. Due to connectedness of $C$, a second one $C_2$ must exist s.t.\ the
        Euclidean distance between $C_1$ and $C_2$ is $0$. But part (a) then implies that also $C_1\cup C_2$ is
        path-connected, a contradiction.
        Moreover, using the statement of part (a) we can transform any curve in $\D_\theta(\nu)$ to a polygonal
        chain which completely lies in $\D_\theta(\nu)$.
      
        Let us turn to the convexity of connected components. Fix a component $C$ of $\D_\theta(\nu)$ and $x,y\in C$,
        s.t.\ $\n{x-y}\geq\theta$, since otherwise (a) guarantees 
        $$\conv(\{x,y\})=\{\alpha x+(1-\alpha)y,\;\alpha\in[0,1]\}\subseteq C.$$
        By the above, there exists a polygonal chain in $\D_\theta(\nu)$, say
        $$l:=\begin{cases}[0,1]\to\R^k\\s\mapsto l(s)\end{cases}$$
        such that $l(0)=x,\ l(1)=y$ and $l$ is continuous and piecewise linear. 
        Let us define $x_0=x,\ x_{j+1}=l(s_j),$ where $s_j:=\max\{s\in[0,1],\;\n{x_j-l(s)}=\tfrac{\theta}{2}\}$, if
        $\n{x_j-y}\geq\theta$ and $x_{j+1}=y$ otherwise. Using (a) and these intermediate points shows that we can assume
        without loss of generality a certain sparseness of the chain, namely that its intermediate
        points $x_1,\dots,x_n$ are s.t.\ pairwise distances in $\{x=x_0,x_1,\dots,x_n,x_{n+1}=y\}$ are at least
        $\tfrac{\theta}{2}$ and hence $n\leq\tfrac{2L}{\theta}$, where $L$ denotes the length of the original chain.
        Note that the modification of the polygonal chain as just described will only decrease its length.
        \par\begingroup \rightskip18.4em
        Given a polygonal chain in $\D_\theta(\nu)$ connecting $x$ and $y$, let us assume that the minimal
        angle at an intermediate point is $\pi-2\alpha<\pi$ at $x_j$.
        Considering $B[x_j,\tfrac{\theta}{2}]$ and using (a) 
        once more, we can replace $x_j$ by the two intersection points of the ball's boundary and the chain
        $x_j^{(1)},x_j^{(2)}$ 
        and conclude that the polygonal chain through the points $x,x_1,\dots,x_{j-1},x_j^{(1)},x_j^{(2)},$\linebreak 
        $x_{j+1},\dots,x_n,y$ still lies in $\D_\theta(\nu)$ and is at least by $\theta\cdot(1-\cos(\alpha))$ shorter.
        \par\endgroup
        \vspace*{-6.3cm}
        \begin{figure}[H]     
        \flushright \includegraphics[scale=1.4]{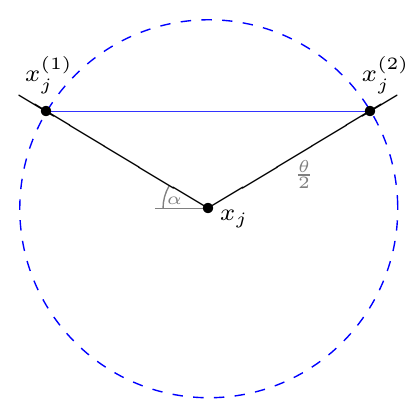}
        \end{figure}
        \vspace*{-0.5cm}
        We can then sparsify the updated chain as described above and denote the result by $l_1$.
        Iterating the whole procedure gives a sequence $(l_m)_{m\in\N}$ of shorter and shorter polygonal chains in
        $\D_\theta(\nu)$ connecting $x$ and $y$. Since the length is bounded below by $\n{x-y}$, the internal angels
        must approach $\pi$ uniformly. Let $\pi-2\alpha_1,\dots,\pi-2\alpha_n$ be the angles at $x_1,\dots,x_n$.
        An easy geometric argument yields that all points on the chain are at distance at most 
        $$\sum_{j=1}^n\tan(2\alpha_1+\dots+2\alpha_j)\,L\leq \tfrac{8nL}{\pi}\sum_{j=1}^n\alpha_j\leq
        \tfrac{16L^2}{\pi\theta}\sum_{j=1}^n\alpha_j.$$
        from the line through $x$ and $x_1$, if $\sum_{j=1}^n\alpha_j\leq\tfrac{\pi}{8}$, as $\tan(z)\leq \tfrac{4}{\pi}z$
        for all $z\in[0,\tfrac{\pi}{4}]$. This also holds for the endpoint $y$, which is why the maximal distance
        of a point on the chain to the line segment between $x$ and $y$ is bounded by
        $\tfrac{32L^2}{\pi\theta}\sum_{j=1}^n\alpha_j$.
        Let $n_m$ and $(\alpha_j^{(m)})_{j=1}^{n_m}$ correspond to $l_m$. Then
        $$\sum_{j=1}^{n_m}\alpha_j^{(m)}\leq\tfrac{2L}{\theta}\max_{1\leq j\leq n_m}\alpha_j^{(m)}
        \stackrel{m\to\infty}{\longrightarrow} 0$$
        implies that the sequence $(l_m)_{m\in\N}$ must approach the line segment between $x$ and $y$, i.e.\ 
        $\conv(\{x,y\})=\{\alpha x+(1-\alpha)y,\;\alpha\in[0,1]\}$, uniformly -- in the sense that
        $$\max_{s\in l_m}\min_{z\in\conv(\{x,y\})}\n{s-z}\to 0\quad\text{as }m\to\infty.$$
        Since $C$ being a component of $\D_\theta(\nu)$ is closed, we find 
        $\conv(\{x,y\})\subseteq C,$ which proves the convexity of C.

        Assuming that there are two points in different connected components, say $x\in C_1, y\in C_2$ s.t.\ 
        $\n{x-y}<\theta$,
        already implies (by part (a)) that $C_1\cup C_2$ is connected, as before. Finally, if $\D_\theta(\nu)$ is
        connected, what we just proved induces that it is convex. Being a closed superset of $\supp(\eta_0)$, this implies
        $$\overline{\conv(\supp(\eta_0))}\subseteq\D_\theta(\nu),$$
        which by Lemma \ref{properties} is all that needed to be shown.
  
  \item Let us now assume that $\nu$ has not only a finite radius but also mass around its mean, that is
        $\E\eta_0\in \supp(\eta_0)$. For $\theta>R$, $\D_\theta(\nu)$ is then
        connected, which by part (c) implies the claim. Indeed, let $\epsilon\in(0,\theta-R)$ and choose a point $x$
        in $B[\E\eta_0,\epsilon]\cap\supp(\eta_0)$. By the choice of $\epsilon$, all points in $B[\E\eta_0,R]$
        are at distance less than $\theta$ from $x$, which by the reasoning in part (a) and 
        $\D_\theta(\nu)\subseteq B[\E\eta_0,R]$ (see Lemma \ref{properties}) implies
        $\conv(\{x,y\})\subseteq \D_\theta(\nu)$ for all $y\in\D_\theta(\nu)$, hence the connectedness of
        $\D_\theta(\nu)$.
     
  \item The first thing to notice is that, given $R<\infty$, for all $\theta>0$ the set $\D_\theta(\nu)$ has
        finitely many connected components. Indeed, choose a point $x_i$ in each, then the open balls $B(x_i,\theta)$ must be
        disjoint by (c) and lie within $B(\E\eta_0,R+\theta)$. Consequently, there can't be more than 
        $(\tfrac{R+\theta}{\theta})^k$ of them.

        Let $C_1,\dots,C_n$ be the connected components of $\D_\delta(\nu)$, for some $\delta>0$, and 
        $d\geq\delta$ the minimal distance between them. When $\theta$ is made larger than $d$, at least two of
        the components merge. Hence there can be only $n-1$ further jumps. For $\delta\leq\theta<d$ we have
        $\D_\theta(\nu)=\D_\delta(\nu)$.
  \item Let us assume the contrary, i.e.\ $\E\eta_0\notin \D_\theta(\nu)$. As this set is closed, there
        exists some $y\in\D_\theta(\nu)$ such that the Euclidean distance from $\E\eta_0$ to $\D_\theta(\nu)$ 
        is given by $\n{\E\eta_0-y}>0$.
        
        \par\begingroup \rightskip18em 
        Choosing $x:=\tfrac12(\E\eta_0+y)$ and using the convexity of $\D_\theta(\nu)$
        -- if there existed $z\in\D_\theta(\nu)$ such that $(z-y)\cdot(x-y)>0$, $y$ would not be closest to $\E\eta_0$
        in $\D_\theta(\nu)$ --
        as well as $\supp(\eta_0)\subseteq \D_\theta(\nu)$ we find\\[0.2cm]\hspace*{0.2cm}
        $\E\big((\eta_0-x)\cdot(y-x)\big)>0$, but\\[0.1cm]\hspace*{0.75cm} $(\E\eta_0-x)\cdot(y-x)<0,$\\[0.2cm]
        a contradiction.
        \par\endgroup
\end{enumerate}
    \vspace*{-5.5cm}
    \begin{figure}[H]
     \hspace{5.6cm} \includegraphics[scale=1.3]{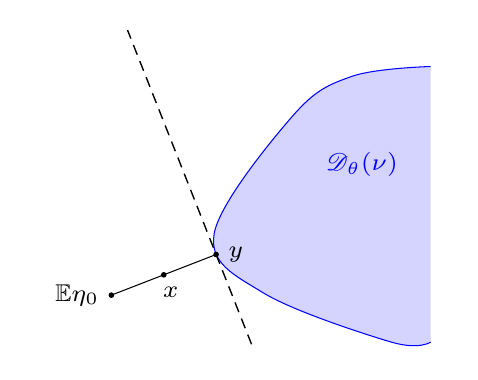}
    \end{figure}  \vspace*{-0.5cm}
\end{proof}

\begin{example}
 \begin{enumerate}[(a)] 
 \item
  To get an impression of how $\D_\theta(\nu)$ grows with $\theta$, let us consider the initial distribution
  on $\R^3$ given by $\text{\upshape unif}(\{(2,1,0),(2,-1,0),(-2,0,1),(-2,0,-1)\})$, i.e.\ featuring four point
  masses at the given vertices. It is easy to check that $\E\eta_0=(0,0,0)$ and $R=\sqrt{5}$, see Figure \ref{convex}.
     
  Since all pairwise distances are at least $2$, $\D_\theta(\nu)=\supp(\eta_0)$ for $\theta<2$.
  For $\theta\geq2$ the opinion values $(2,1,0)$ and $(2,-1,0)$ can compromise, same for
  $(-2,0,1)$ and $(-2,0,-1)$. This implies that $\D_\theta(\nu)$ contains both line segments 
  $\{(2,\alpha,0),\;\alpha\in[-1,1]\}$ and $\{(-2,0,\alpha),\;\alpha\in[-1,1]\}$.
  The latter are at distance $4$, hence we can conclude
  $$\D_\theta(\nu)=\begin{cases}
   \{(2,1,0),(2,-1,0),(-2,0,1),(-2,0,-1)\},&\!\text{for }\theta<2\\
   \{(2,\alpha,0),(-2,0,\alpha),\;\alpha\in[-1,1]\},&\!\text{for }\theta\in[2,4)\\
   \conv(\{(2,1,0),(2,-1,0),(-2,0,1),(-2,0,-1)\}),&\!\text{for }\theta>4.
   \end{cases}$$
  \begin{figure}[H]
     \centering
     \includegraphics[scale=1]{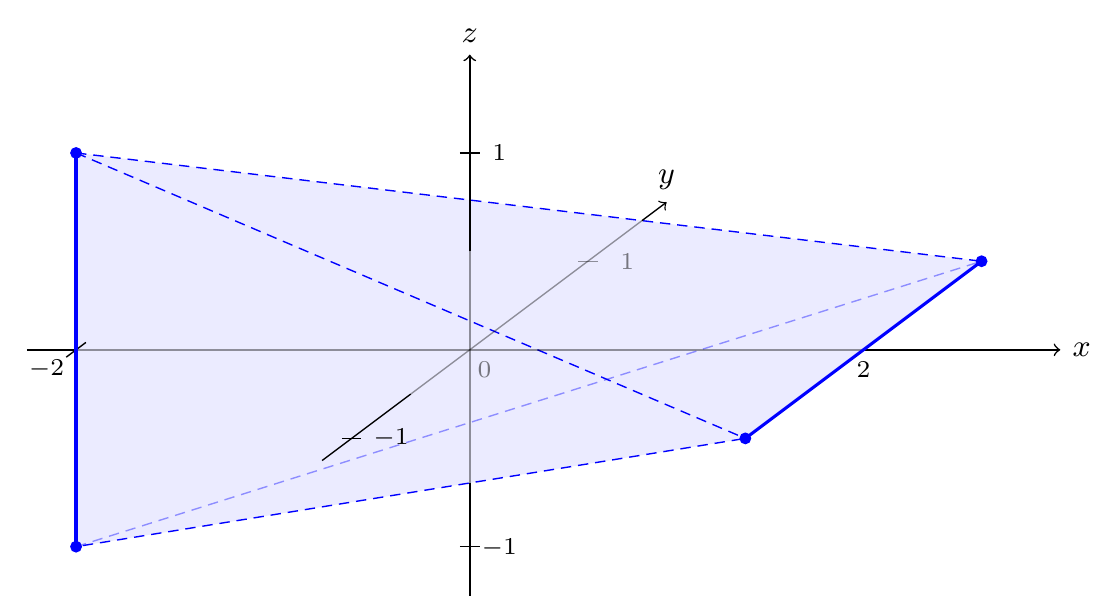}
     \caption{$\D_\theta(\nu)$ for $\eta_0$ being uniformly distributed on the set\\
     $\{(2,1,0),(2,-1,0),(-2,0,1),(-2,0,-1)\}$, evolving with growing $\theta$.}\label{convex}
  \end{figure}
  \noindent  
  For $\theta=4$ it depends on whether the values $(-2,0,0),(2,0,0)$ can be achieved or merely approximated
  by finite configurations, in other words $\mu$ (see also Example \ref{mu}). Note how $\D_\theta(\nu)$ grows
  by forming local convex hulls. 
  
  If we choose $\text{\upshape unif}(\{(0.99,1,0),(0.99,-1,0),(-0.99,0,1),(-0.99,0,-1)\})$ to be the initial
  distribution instead, we can observe a certain chain reaction effect. $\theta\geq2$ brings the point masses
  pairwise within the confidence bound as before, but this time also their convex hulls. So for this distribution
  $\nu$ we find
  $$\D_\theta(\nu)=\begin{cases}
   \supp(\eta_0),&\text{for }\theta<2\\
   \conv(\supp(\eta_0)),&\text{for }\theta\geq2.
   \end{cases}$$
 \item
  Example \ref{jump} already shows that the mapping $\vartheta\mapsto \D_\vartheta(\nu)$ can have infinitely
  (but still countably) many jumps on $(0,\infty)$. Taking the discrete initial distribution given by
  $$\Prob(\eta_0=2^n)=\tfrac{1}{3^{n}}\text{ and }\Prob(\eta_0=-2^n)=\tfrac{1}{3^{n}},\text{ for }n\in\N,$$
  shows that part (e) of Lemma \ref{circle} doesn't hold for the case $R=\infty$, i.e.\ under the
  weaker condition that $\E\eta_0$ is finite.
 \item
  Coming back to the example mentioned above, where $\eta_0\sim\text{\upshape unif}(S^{k-1})$ for some $k\geq 2$, it is
  not hard to see that $\D_\theta(\nu)=B[\mathbf{0},1]$ for all $\theta>0$. Indeed, since 
  $\supp(\eta_0)=S^{k-1}$ is connected and $\supp(\eta_0)\subseteq\D_\theta(\nu)$, it has to be contained in a
  connected component of $\D_\theta(\nu)$. All such are convex by Lemma \ref{circle}, hence 
  $\conv(S^{k-1})=B[\mathbf{0},1]\subseteq\D_\theta(\nu)$. The reverse inclusion follows directly from part (b)
  of Lemma \ref{properties}.
\end{enumerate}
\end{example}

\begin{definition}\label{suppdef}
For $\theta>0$ and $t\geq0$, let the {\em support of the distribution of $\eta_t$} be denoted by
$\supp_\theta(\eta_t)$. 
\end{definition}

\noindent
The support of $\eta_t$ evidently depends on $\theta$. However, for $t=0$ it holds that
$\supp_\theta(\eta_0)=\supp(\eta_0)$
irrespectively of $\theta$, as the dynamics of the model is not yet involved. Note that for values of $\theta$ where
$\D_\theta(\nu)$ increases, $\supp_\theta(\eta_t)$ can actually depend on $\mu$ as well, see Example \ref{mu} below.
Let us next derive properties of $\supp_\theta(\eta_t)$ similar to those of $\D_\theta(\nu)$.
\begin{lemma}\label{suppt}
 \begin{enumerate}[(a)]
  \item For $0<s<t$ we get $\supp_\theta(\eta_s)=\supp_\theta(\eta_t)$.
  \item $\supp_\theta(\eta_t)$ increases with $\theta$ and for all $\theta>0$:
        $$\supp(\eta_0)\subseteq \supp_\theta(\eta_t)\subseteq\overline{\conv(\supp(\eta_0))}\subseteq B[\E\eta_0,R].$$
 \end{enumerate}
\end{lemma}

\begin{proof}
 \begin{enumerate}[(a)]
  \item $\supp_\theta(\eta_s)\subseteq\supp_\theta(\eta_t)$ readily follows from the fact, that for every set $A$
        $\Prob(\eta_s(v)\in A)>0$ implies $\Prob(\eta_t(v)\in A)>0$, since with positive probability
        there won't be any Poisson events on the edges $\langle v-1,v\rangle$ and $\langle v,v+1\rangle$ in the time
        interval $[s,t]$ forcing $\eta_s(v)=\eta_t(v)$.
        
        But the reverse inclusion is also true. To see this we will locally modify the configuration:
        $x\in\supp_\theta(\eta_t)$ if and only if for all $r>0$, there exists
        some $n\in\N$ such that the event that $\eta_t(0)\in B[x,r]$ and at least one of the edges $\langle -n,-n+1\rangle,
        \dots,\langle -1,0\rangle$ and $\langle 0,1\rangle,\dots,\langle n-1,n\rangle$ respectively, has not experienced
        any Poisson event up to time $t$ has positive probability. That the Poisson events occurring on
        $\langle -n,-n+1\rangle,\dots,\langle n-1,n\rangle$ up to $t$ already occur in the same order up to time $s$
        (and no further events) has positive probability. Due to the fact that the Poisson events are independent of
        the starting configuration, such a modification of the interactions shows $\Prob(\eta_s(0)\in B[x,r])>0$.
  \item To prove the monotonicity in $\theta$, we will dissect the event described in part (a) a little more closely.
        For $x\in\supp_\vartheta(\eta_t)$ and $r>0$, let us consider the event that $\eta_t(0)\in B[x,r]$ and at least one
        of the edges between $-n$ and $0$ as well as between $0$ and $n$ has not experienced any Poisson event up to time
        $t$. For sufficiently large $n$ this has positive probability as mentioned before. Fix $n$ to be large enough in this
        respect and denote the corresponding event by $A$.
        
        Let again $(e_i)_{i=1}^N$ encode the chronologically ordered locations of the random but finite number of
        Poisson events occurring up to time $t$ on the edge set $\langle -n,-n+1\rangle,\dots,\langle n-1,n\rangle$.
        Further, let $(e_{i_j})_{j=1}^{N'}$ be the subsequence of $(e_i)_{i=1}^N$ which contains only those edges on which a
        difference exceeding the confidence bound prevented the occurring Poisson event from invoking an actual update
        of opinions.
        Since there are only finitely many choices for the sequence $(e_i)_{i=1}^N$ and its corresponding subsequence, if
        $N\in\N$ is fixed, and $N$ is a.s.\ finite, we can partition the event $A$ into $\{A_m,\;m\in\N\}$ according to the
        different choices of $(e_i)$ and $(e_{i_j})$. Note that for the subsequences to be considered equal not only
        their length and ordered elements must coincide, but also the set of indices $\{i_j,\;1\leq j\leq N'\}$ has to
        be identical. From $\Prob(A)>0$ we can conclude that there must be some $A_m$ which has positive probability.
        In other words, there exists a set $C\subseteq(\R^k)^{2n-1}$ s.t.\ 
        $$\Prob\big((\eta_0(v))_{v=-n+1}^{n-1}\in C\big)>0$$
        and given a starting configuration in $C$, Poisson events on the edges given by the fixed sequence $(e_i)_{i=1}^N$
        corresponding to $A_m$ will ensure, in the Deffuant model with confidence bound $\vartheta$, that the final value
        at $0$ is in $B[x,r]$.
        
        Let $B$ be the event that the locations of all Poisson events on the edge set
        $\{\langle -n,-n+1\rangle,\dots,\langle n-1,n\rangle\}$ up to $t$ are given by the subsequence of $(e_i)_{i=1}^N$
        which is obtained by removing the elements of $(e_{i_j})$.
        Given $B$ and $\{(\eta_0(v))_{v=-n+1}^{n-1}\in C\}$, the dynamics of the Deffuant model with confidence
        bounds $\vartheta$ and $\theta\geq\vartheta$ respectively will coincide up to time $t$ between the two edges
        without Poisson events shielding $0$ from $-n$ and $n$. Since $B$ has positive probability and the Poisson
        events are independent of $\{(\eta_0(v))_{v=-n+1}^{n-1}\in C\}$ this implies that $x\in\supp_\vartheta(\eta_t)$
        forces $x\in\supp_{\theta}(\eta_t)$ for all $\theta\geq\vartheta$, hence the claimed monotonicity.
               
        When it comes to the second statement, the first inclusion was actually proved in (a) as the argument
        used in order to show $\supp_\theta(\eta_s)\subseteq\supp_\theta(\eta_t)$ is also valid for $s=0$.
        The second and third inclusion can be verified as in part (b) of Lemma \ref{properties}.
 \end{enumerate}
\end{proof}

\noindent
The following proposition reveals how the set $\D_\theta(\nu)$ comes into play in the analysis of the
long-term behavior of the Deffuant model.
\begin{proposition}\label{supp_t}
   If $\vartheta\mapsto \D_\vartheta(\nu)$ has no jump in $[\theta-\epsilon,\theta+\epsilon]$ for fixed
   $\theta$ and some $\epsilon>0$, the following equality holds true for all $t>0$:
   $$\supp_{\theta}(\eta_t)=\D_\theta(\nu).$$
\end{proposition}

\begin{proof}
Before proving this result, we want to mention that given $R<\infty$, the continuity assumption can be
weakened: If $R<\infty$ and $\vartheta\mapsto\D_\vartheta(\nu)$ has no jump at $\theta$, part (e)
of Lemma \ref{circle}, already implies that $\D_\vartheta(\nu)$ is constant on an interval $[\theta-\epsilon,\theta+\epsilon]$ for suitably small $\epsilon>0$. 

Let us first focus on the inclusion $\supp_{\theta}(\eta_t)\supseteq\D_\theta(\nu)$. For every fixed
$x$ in $\D_\theta(\nu)=\D_{\theta-\epsilon}(\nu)$ and all $r>0$, there exists a finite
configuration with $n\in\N$, $x_1,\dots,x_n\in\supp(\eta_0)$ and edge sequence $(e_i)_{i=1}^N$, s.t.\
updating the configuration with respect to the confidence bound $\theta-\epsilon$ yields $\eta_N(v)\in B[x,r]$ for
all $v\in\{1,\dots,n\}$. Let further $t>0$ be fixed.
Due to $x_v\in\supp(\eta_0)$, we get $\Prob(\eta_0\in B[x_v,\epsilon])>0$.

Consequently, in the Deffuant model on $\Z$ the following event has positive probability:
$\eta_0(v)\in B[x_v,\epsilon]$ for all $v\in\{1,\dots,n\}$, up to time $t$ Poisson events have occurred
on neither $\langle0,1\rangle$ nor $\langle n,n+1\rangle$ and the locations of the events on $\langle1,2\rangle,\dots,\langle n-1,n\rangle$ are chronologically ordered given by $(e_i)_{i=1}^N$.
Note that every Poisson event
which leads to an update in the given finite configuration does the same in this configuration of the whole
model with respect to parameter $\theta$, as the margins coming from slightly altered initial values are
convex combinations of the initial margins $\eta_0(v)-x_v$ and thus always bounded by $\epsilon$.
This shows $\Prob(\eta_t(1)\in B[x,r+\epsilon])>0$, hence $x\in \supp_{\theta}(\eta_t)$.

When it comes to the reverse inclusion, consider again the Deffuant model with confidence bound
$\theta$. By definition, $x\in\supp_{\theta}(\eta_t)$ if and only if for all $r>0:$ 
$\Prob(\eta_t(v)\in B[x,r])>0$. But every such value $\eta_t(v)$ is formed by
(finitely many) convex combinations starting from a finite collection of initial values $\{\eta_0(u)\}_{u=v-k}^{v+l}$.
Part (a) of Lemma \ref{circle} shows that $\eta_{s-}(u),\eta_{s-}(v)\in\D_{\theta+\epsilon}(\nu)$ implies 
$\eta_{s}(u),\eta_{s}(v)\in\D_{\theta+\epsilon}(\nu)$ after an update along the edge $\langle u,v\rangle$
at time $s$, since this can only occur if the former are at distance less than or equal to $\theta$. 
Thus, due to $\{\eta_0(u)\}_{u=v-k}^{v+l}\subseteq\supp(\eta_0)\subseteq\D_{\theta+\epsilon}(\nu)$, an inductive
argument verifies $\eta_t(v)\in\D_{\theta+\epsilon}(\nu)$ and hence
$$\supp_{\theta}(\eta_t)\subseteq\overline{\D_{\theta+\epsilon}(\nu)}=\D_{\theta+\epsilon}(\nu)=\D_\theta(\nu).
\vspace{-1.5em}$$
\end{proof}
\\[1em]\noindent
Note that if $\vartheta\mapsto \D_\vartheta(\nu)$ has a jump at $\theta$, the subtle issue with critical
compromises, as considered in Proposition \ref{crit}, reappears. To make this point clear, let us consider
the initial distribution $\nu=\text{\upshape unif}(\{\tfrac14,\tfrac34\})$, for which we find
$$\D_{\tfrac12}(\nu)=\supp_{\tfrac12}(\eta_t)=[\tfrac14,\tfrac34].$$
Taking $\eta_0\sim\text{\upshape unif}\big([0,\tfrac14]\cup[\tfrac34,1]\big)$ instead yields
$$[0,1]=\D_{\tfrac12}(\nu)\supsetneq\supp_{\tfrac12}(\eta_t)=[0,\tfrac14]\cup[\tfrac34,1].$$

\begin{definition}\label{gap}
Given an initial distribution $\mathcal{L}(\eta_0)=\nu$, define the length of the {\em largest gap} in its
support as
$$h:=\inf\{\theta>0,\;\D_\theta(\nu)\text{ is connected}\}.$$
\end{definition}
Following this definition we get $h=0$ for $\nu=\text{\upshape unif}(S^{k-1})$ and $k\geq2$,
but $h=2$ for $\nu=\text{\upshape unif}(S^0)$. Considering the other two distributions appearing
in the above example, we observe that $\text{\upshape unif}(\{(2,1,0),(2,-1,0),(-2,0,1),(-2,0,-1)\})$
has $h=4$ and $\text{\upshape unif}(\{(0.99,1,0),(0.99,-1,0),(-0.99,0,1),(-0.99,0,-1)\})$ instead $h=2$.
In addition, parts (b) and (d) of Lemma \ref{circle} tell us that $h\leq 2R$ if $R$ is finite and $h\leq R$
if additionally $\E\eta_0\in\supp(\eta_0)$.

Having generalized the notion of a gap in a distribution on $\R$ to higher dimensions finally allows us to formulate
and prove a result corresponding to the cases of Theorem \ref{gen} that were omitted by Theorem \ref{nogap}.

\begin{theorem}\label{gapsEucl}
Consider the Deffuant model on $\Z$ with an initial distribution on $(\R^k,\n{\,.\,})$ that is
bounded, i.e.\
$$R=\inf\left\{r>0,\;\Prob\big(\eta_0\in B[\E\eta_0,r]\big)=1\right\}<\infty,$$
and $h$ being the length of the largest gap in its support.
Then the critical value for the confidence bound, where a phase transition from a.s.\ no consensus
to a.s.\ strong consensus takes place is $\theta_\text{\upshape c}=\max\{R,h\}$.
\end{theorem}

\begin{proof}
Having analyzed the qualitative differences invoked by higher-dimen\-sional opinion values, the proof of this
theorem is to a large extent similar to the one of part (a) of Thm.\ 2.2 in \cite{Deffuant}, which is Theorem
\ref{gen} in the foregoing section. Let us consider the following three scenarios:

\begin{enumerate}[(i)]
\item {\em For $\theta<h$ we cannot have consensus:}\\
By definition of $h$ the set $\D_{\theta+\epsilon}(\nu)$ is not
connected for $\epsilon>0$ sufficiently small; by Lemma \ref{circle} (e) we can choose $\epsilon$ such that
$\vartheta\mapsto \D_\vartheta(\nu)$ has no jump at $\theta+\epsilon$ and thus (by Proposition \ref{supp_t}) get
$\D_{\theta+\epsilon}(\nu)=\supp_{\theta+\epsilon}(\eta_t)$ for all $t>0$. In addition, Lemma \ref{circle} (c)
tells us that there exist two connected components, say $C_1$ and $C_2$, both being convex and at distance at least
$\theta+\epsilon$ from the corresponding complementary part of $\supp_{\theta+\epsilon}(\eta_t)$,
i.e.\ $\n{x-y}\geq\theta+\epsilon$ for all $x\in C_i, y\in \supp_{\theta+\epsilon}(\eta_t)\setminus C_i$
and $i=1,2$.

By Lemma \ref{suppt} we know that $\supp(\eta_0)\subseteq\supp_{\theta}(\eta_t)\subseteq\supp_{\theta+\epsilon}(\eta_t)$.
In the Deffuant model with confidence bound $\theta$ opinions in $C_1$ cannot compromise with opinions in
$\supp_{\theta}(\eta_t)\setminus C_1\subseteq\supp_{\theta+\epsilon}(\eta_t)\setminus C_1$ and thus never leave
the convex set $C_1$. The same holds for $C_2$.

Consequently, $\Prob(\eta_0(v)\in C_i)=\Prob(\eta_t(v)\in C_i)>0$, for $i=1,2$. For a fixed vertex $v$, it follows
from the independence of initial opinions that $\Prob(\eta_0(v)\in C_1,\eta_0(v+1)\in C_2)>0$, which dooms the edge
$\langle v, v+1\rangle$ to be blocked for all $t\geq0$, due to $\n{\eta_t(v)-\eta_t(v+1)}\geq\theta+\epsilon$.
Ergodicity of the initial configuration ensures that a.s.\ infinitely many neighboring vertices will be prevented from
compromising by holding opinions in $C_1$ and $C_2$ respectively, hence no consensus in the long run.
    \vspace{0.3cm}
    \begin{figure}[H]
     \hspace{6.6cm} \includegraphics[scale=1.1]{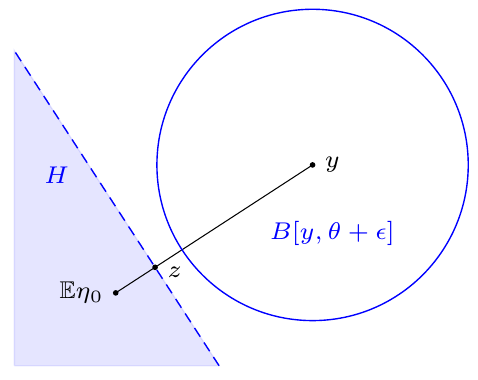}
    \end{figure}  
    \vspace*{-5.5cm}
\item {\em For $\theta<R$ we cannot have consensus:}
      \par\begingroup \rightskip17em
      Given $\theta<R$, there exists some $y\in\supp(\eta_0)\setminus B[\E\eta_0,\theta+2\epsilon]$ for fixed
      $\epsilon\in\big(0,\tfrac{R-\theta}{2}\big)$. Choose $z$ to be the point on the line segment connecting
      $\E\eta_0$ and $y$ which has Euclidean distance $\epsilon$ to $\E\eta_0$, see the picture to the right.
      With help of this point, define the half-space $H:=\{x\in\R^k,\; (x-z)\cdot(y-z)\leq0\}$. 
      Clearly, $B[\E\eta_0,\epsilon]\subseteq H$ and by the
      same\hfill argument\hfill as\hfill in\hfill part\hfill (e)\hfill of
      \par\endgroup\vspace*{-0.14cm}
      Lemma \ref{circle}: $\Prob(\eta_0\in H)>0$, as the contrary would imply
      $$\E[(\eta_0-z)\cdot(y-z)]>0>(\E\eta_0-z)\cdot(y-z),$$
      a contradiction.
      
      Using this auxiliary construction, we can finish the proof of this subcase following the argument in the proof
      of Theorem \ref{gen} (b), see Thm.\ 2.2 in \cite{Deffuant}. As the distribution is bounded, the SLLN states
      \begin{equation}\label{SLLN}
       \Prob\left(\lim_{n\to\infty}\frac{1}{n}\sum_{u=v+1}^{v+n}\eta_0(u)=\E\eta_0\right)=1.
      \end{equation}
      Consequently, for sufficiently large $N\in\N$ the following event has non-zero probability:
      \begin{equation*}
       A_N:=\left\{\frac{1}{n}\sum_{u=v+1}^{v+n}\eta_0(u)\in H\text{ for all }n\geq N\right\}.
      \end{equation*}
      Let $\xi$ denote the (real-valued) distribution of $(\eta_0-z)\cdot(y-z)$ and $\xi|_{(-\infty,0]}$ its
      distribution conditioned on the event $\{(\eta_0-z)\cdot(y-z)\leq0\}=\{\eta_0\in H\}$. Obviously, 
      $\xi|_{(-\infty,0]}$ is stochastically dominated by $\xi$, i.e.\ $\xi|_{(-\infty,0]}\preceq\xi$, which implies
      $$\left(\bigotimes_{u=v+1}^{v+N}\xi|_{(-\infty,0]}\right)\otimes\left(\bigotimes_{u>v+N}\xi\right)
      \preceq\bigotimes_{u\geq v+1}\xi .$$
      
      Let $B$ be the event $\{\eta_0(v+1)\in H,\dots,\eta_0(v+N)\in H\}$, which has non-zero probability
      by independence, and
      $$A_1:=\left\{\frac{1}{n}\sum_{u=v+1}^{v+n}\eta_0(u)\in H\text{ for all }n\in\N\right\}.$$
      Rewriting the event $A_N$ as
      \begin{equation*}
       A_N=\left\{\frac{1}{n}\sum_{u=v+1}^{v+n}\big(\eta_0(u)-z\big)\cdot\big(y-z\big)\leq0\text{ for all }n\geq N\right\},
      \end{equation*}
      the stochastic domination from above yields:
      \begin{align*}
       \Prob(A_1)&\geq\Prob(A_1\cap B)=\Prob(A_N\cap B)=\Prob(A_N|B)\cdot\Prob(B)\\
                 &\geq\Prob(A_N)\cdot\Prob(B)>0.
      \end{align*}
      The very same ideas as in the proof of Prop.\ 5.1 in \cite{ShareDrink} show that if $A_1$ occurs and the edge
      $\langle v,v+1\rangle$ doesn't allow for an update up to time $t>0$, irrespectively of the dynamics on
      $\{u\in\Z, u\geq v+1\}$, we get that $\eta_t(v+1)$ is a convex combination of the averages
      $\{\frac{1}{n}\sum_{u=v+1}^{v+n}\eta_0(u),\; n\in\N\}$, hence in $H$ as the latter is convex.
      By symmetry, the same holds for site $v-1$ and the half-line to the left, i.e.\ $\{u\in\Z, u\leq v-1\}$.
      Independence of the initial opinions therefore guarantees that with positive probability, the initial
      configuration can be such that $\eta_0(v)\in B(y,\epsilon)$ and the values at sites $v-1$ and $v+1$ are doomed
      to stay in $H$, blocking the edges adjacent to $v$ once and for all, as the distance of $y$ to $H$ is at least
      $\theta+\epsilon$. Ergodicity makes sure that with probability $1$ infinitely many sites will get stuck
      this way.

\item {\em For $\theta>\max\{R,h\}$ we get a.s.\ strong consensus:}\\
Choose $\beta$ such that $0<\beta<\theta-\max\{R,h\}$. By definition of $h$ and Lemma \ref{circle} (e),
$\E\eta_0\in\D_{\theta-\beta}(\nu)$. Because of that, for all $\epsilon>0$, there exists a finite configuration
such that the final opinion values all lie in $B[\E\eta_0,\tfrac{\epsilon}{6}]$,
i.e.\ $n\in\N$, $x_1,\dots,x_n\in\supp(\eta_0)$ and an edge sequence $(e_i)_{i=1}^N$ from 
$\{\langle1,2\rangle,\dots,\langle n-1,n\rangle\}$, s.t.\ updating the configuration with respect to the confidence bound
$\theta-\beta$ yields $\eta_N(v)\in B[\E\eta_0,\tfrac{\epsilon}{6}]$ for all $v\in\{1,\dots,n\}$, see Definition
\ref{Dtheta}. From this point on, we can go about as in step (ii) of the proof of Thm.\ 2.2 (a) in \cite{Deffuant}:

      Let us consider some fixed time point $t>0$ and the corresponding configuration $\{\eta_t(v)\}_{v\in\Z}$.
      With probability 1, there exists an infinite increasing sequence of not necessarily consecutive edges
      $(\langle v_k,v_k+1\rangle)_{k\in\N}$ to the right of site $1$, on which no Poisson event has occurred up to
      time $t$.
      
      Let $l_k:=v_{k+1}-v_k,\text{ for } k\in\N,$ denote the random lengths of the intervals in between and
      $l_0:=v_1-v_0+1$ the one of the interval including $1$, where $\langle v_0-1,v_0\rangle$ is the first 
      edge to the left of $1$ without Poisson event. Since the involved Poisson processes are independent,
      it is easy to verify that the $l_k,\ k\in\N_0=\{0,1,2,\dots\}$, are i.i.d., having a geometric distribution
      on $\N$ with parameter $\text{e}^{-t}$.
      
      For $\delta>0$, let $A_\delta$ be the event that $l_0$ is finite and only finitely many of the events
      $\{l_k\geq k\,\tfrac{\delta}{R}\},\ k\in\N,$ occur. Then their independence and the Borel-Cantelli lemma tell
      us that $A_\delta$ has probability $1$. On $A_\delta$ however the following holds a.s.\ true:\vspace{-0.5cm}
      
      \begin{align*}
       \limsup_{v\to\infty}\Big\lVert\frac{1}{v}\sum_{u=1}^{v}\eta_t(u)-\E\eta_0\Big\rVert_2
       &=\limsup_{v\to\infty}\Big\lVert\frac{1}{v}\sum_{u=1}^{v}\big(\eta_t(u)-\E\eta_0\big)\Big\rVert_2\\
       &=\limsup_{v\to\infty}\Big\lVert\frac{1}{v}\sum_{u=v_0}^{v}\big(\eta_t(u)-\E\eta_0\big)\Big\rVert_2\\
       &\leq\limsup_{v\to\infty}\Big\lVert\frac{1}{v}\sum_{u=v_0}^{v}\big(\eta_0(u)-\E\eta_0\big)\Big\rVert_2
       +\delta\\
       &=\limsup_{v\to\infty}\Big\lVert\frac{1}{v}\sum_{u=1}^{v}\big(\eta_0(u)-\E\eta_0\big)\Big\rVert_2
       +\delta\\
       &=\delta.
      \end{align*}
 
      The second and second to last equality follow from the finiteness of $v_0$, the last equality from
      the SLLN applied to the sequence $(\eta_0(u))_{u\geq1}$, stating
      $$\lim_{v\to\infty}\frac{1}{v}\sum_{u=1}^{v}\eta_0(u)=\E\eta_0\text{ almost surely.}$$
      The inequality is due to the fact that the Deffuant model is mass-preserving in the sense that
      $\eta_t(u)+\eta_t(v)=\eta_{t-}(u)+\eta_{t-}(v)$ in (\ref{dynamics}), hence for all $k\in\N$:
      $\sum_{u=v_0}^{v_k}\eta_0(u)=\sum_{u=v_0}^{v_k}\eta_t(u)$. For the average at time $t$ running from $v_0$
      to some $v\in\{v_k+1,\dots, v_{k+1}\}$ to differ by more than $\delta$ from the one at time 0, the interval
      has to be of length more than $k\,\tfrac{\delta}{R}$, since $v_k\geq k$ and $\n{\eta_t(u)-\E\eta_0}\in[0,R]$ for
      all $t,u$. This, however, will happen only finitely many times.
      
      Since $\delta>0$ was arbitrary, we have established that even for $t>0$
      \begin{equation}\label{outerpart}
       \lim_{v\to\infty}\frac{1}{v}\sum_{u=1}^{v}\eta_t(u)=\E\eta_0 \text{ almost surely.}
      \end{equation}
      Now we are going to use the finite configuration from above and a conditional version of the so-called
      {\em local modification}, a technique often used in percolation theory.
      Due to (\ref{outerpart}), there exists some integer number $k$ s.t.\ the event
      $$A:=\left\{\frac{1}{v}\sum_{u=1}^{v}\eta_t(u)\in B[\E\eta_0,\tfrac{\epsilon}{3}]\text{ for all }v\geq kn\right\}$$
      has probability greater than $1-\text{e}^{-2t}$.
      
      Let $B$ in turn be the event that there was no Poisson event on $\langle 0,1\rangle$ and 
      $\langle kn,kn+1\rangle$ up to time $t$, hence $\Prob(B)=\text{e}^{-2t}$. Finally, let $C$ be the event
      that the initial values satisfy
      $$\eta_0(ln+i)\in B[x_i,\min\{\beta,\tfrac{\epsilon}{6}\}],\text{ for all }0\leq l\leq k-1\text{ and }1\leq i\leq n,$$
      and the Poisson firings on the edges $\langle 0,1\rangle,\dots,\langle kn,kn+1\rangle$ up to time $t$ are given
      by a concatenation of the $k$ finite sequences given by shifting $(e_i)_{i=1}^N$ $ln$ vertices to the right,
      $0\leq l\leq k-1$. In other words, up to time $t$ there are no Poisson events on the $k+1$ edges
      $\{\langle0,1\rangle,\langle n,n+1\rangle,\dots,\langle kn,kn+1\rangle\}$ and the dynamics in the $k$ blocks
      $\{ln+1,\dots,(l+1)n\}$ resembles the dynamics of the finite configuration, accordingly leading to
      $\eta_t(v)\in B[\E\eta_0,\tfrac{\epsilon}{3}]$ for all $v\in\{1,\dots,kn\}$, see also the proof of
      Proposition \ref{supp_t}. Note that $C$ has non-zero probability, $C\subseteq B$ and also $A\cap B$ has strictly
      positive probability as $\Prob(A\cap B^\mathsf{c})\leq\Prob(B^\mathsf{c})=1-\text{e}^{-2t}<\Prob(A)$.
      
      Consider two configurations $\{\eta_0'(v)\}_{v\in\Z}$ and $\{\eta_0''(v)\}_{v\in\Z}$, independent from each other and
      having the same distribution as $\{\eta_0(v)\}_{v\in\Z}$ underlying the dynamics of the Deffuant model. Then also
      the compound configuration
      $$\tilde{\eta}_0(v)=\begin{cases}\eta_0'(v),&\text{for }v\in\{1,\dots,kn\}\\
      																 \eta_0''(v),&\text{for }v\notin\{1,\dots,kn\}\end{cases}$$
      has the i.i.d.\ distribution of the initial configuration.
      With positive probability $A\cap B$ occurs for the initial configuration $\{\eta_0''(v)\}_{v\in\Z}$ and
      $C$ for the initial configuration $\{\eta_0'(v)\}_{v\in\Z}$. The fact that $(\tilde{\eta}_s(v))_{v\in\Z}$
      equals $\{\eta_s'(v)\}_{v\in\Z}$ on $\{1,\dots,kn\}$ and $\{\eta_s''(v)\}_{v\in\Z}$ outside $\{1,\dots,kn\}$ for
      $s\in[0,t]$ given $B$, together with the independence of the involved building block configurations, shows
      that with positive probability $A\cap B\cap C'$ holds for the configuration at time $t$, where
      $$C'=\left\{\eta_t(v)\in B[\E\eta_0,\tfrac{\epsilon}{3}]\text{ for all }v\in\{1,\dots,kn\}\right\}.$$
      
      An easy calculation reveals that $A\cap C'$ implies the $\epsilon$-flatness
      to the right of site $1$ in the configuration at time $t$. By symmetry in left and right, the same holds true
      for the site $0$ and $\epsilon$-flatness to the left with respect to the configuration $\{\eta_t(v)\}_{v\in\Z}$.
      As the two parts $\{\eta_t(v)\}_{v\leq0}$ and $\{\eta_t(v)\}_{v\geq1}$ of the configuration at time $t$ are
      conditionally independent given there was no Poisson event on the edge $\langle0,1\rangle$ up to time $t$, we
      have actually shown that the origin is two-sidedly $\epsilon$-flat with respect to the configuration
      $\{\eta_t(v)\}_{v\in\Z}$ with positive probability.
       
      The supercritical case is now settled as in part (a) of Theorem \ref{nogap}. Following the reasoning of
      Sect.\ 6 in \cite{ShareDrink}, the proof of La.\ 6.3 there tells us that a two-sidedly $\epsilon$-flat vertex will
      never move further than $6\epsilon$ away from the mean and Prop.\ 6.1 guarantees that two neighbors will a.s.\ 
      either finally concur or end up further than $\theta$ apart from each other. Choosing 
      $0<\epsilon<\tfrac{\theta-R}{6}$ the latter is impossible for vertices neighboring a two-sidedly $\epsilon$-flat
      vertex, which means that they will a.s.\ finally concur and the same holds true for every vertex by induction.
      Ergodicity of the setting at time $t$ guarantees that there will be a.s.\ (infinitely many) two-sidedly 
      $\epsilon$-flat vertices forcing almost sure strong consensus.
\end{enumerate}\vspace*{-0.9em}
\end{proof}

\begin{remark}
It is worth emphasizing that only the support and expected value of a bounded initial distribution determine
the critical value for $\theta$: As long as it does not affect the support, the dependence relations between the
coordinates of the random vector $\eta_0$ do not influence the critical parameter $\theta_\text{\upshape c}$.

Furthermore, having proved this result for more general multivariate distributions, part (a) of Theorem \ref{nogap}
becomes a special case of Theorem \ref{gapsEucl}, since using part (d) of Lemma \ref{circle} shows that the maximal gap
in a distribution of $\eta_0$ with mass around its mean cannot be larger than its radius, i.e.\ $h\leq R$.\\[0.5em]
\noindent
Finally, the requirement that the initial opinions are independent is not as vital as it might seem. The independence
was merely used to guarantee that we can locally modify initial configurations and still obtain events with positive
probability. Consequently, the i.i.d.\ property can be replaced by the weaker condition that $\{\eta_0(v)\}_{v\in\Z}$
is a stationary sequence, ergodic with respect to shifts and allowing conditional probabilities such that the
conditional distribution of $\eta_0(0)$ given $\{\eta_0(v)\}_{v\in\Z\setminus\{0\}}$ almost surely has the same support
as the marginal distribution $\mathcal{L}(\eta_0)$, with the above conclusions remaining valid. This last condition
is a natural extension to continuous state spaces of the well-known {\em finite energy condition} from percolation theory 
-- for a more detailed discussion of this extension to dependent initial opinions, see Sect.\ 2.2 in \cite{Deffuant}.
\end{remark}

\begin{example}\label{sphere}
\begin{enumerate}[(a)]
\item With Theorem \ref{gapsEucl} in hand, we can finally settle the case of $\eta_0\sim\text{\upshape unif}(S^{k-1})$.
Irrespectively of $k$, this distribution has radius $R=1$, but for $k=1$, the maximal gap is $h=2$, for $k>1$ instead
$h=0$. By the above theorem, we can conclude
$$\theta_\text{\upshape c}=\max\{R,h\}=\begin{cases}
   2,&\text{for }k=1\\
   1,&\text{for }k\geq2.
   \end{cases}$$
 In short, the fact that $S^{k-1}$ is disconnected for $k=1$ but connected for $k\geq2$ makes all the difference. 
\item If the random vector $\eta_0$ has independent coordinates, each being Bernoulli distributed with parameter
$p\in(0,1)$, i.e.\ for all $1\leq i\leq k$
$$\Prob\big(\eta_0^{(i)}=1\big)=1-\Prob\big(\eta_0^{(i)}=0\big)=p,$$ its support is the hypercube $\{0,1\}^k$ and
the expected value $\E\eta_0=p\,\mathbf{e}$, where $\mathbf{e}$ is the $k$-dimensional vector of all ones.
The radius of this initial distribution is
$R=\max\{\n{\E\eta_0-\mathbf{0}},\n{\E\eta_0-\mathbf{e}}\}=\sqrt{k}\,\max\{p,1-p\}.$
It is not hard to see that a distribution with the hypercube as its support has the maximal gap $h=1$. Indeed,
for $\theta<1$ no two opinion values can interact, for $\theta>1$ all neighboring corners get within the confidence
bound and their pairwise convex hulls form the edges of the hypercube, hence their union is a connected set giving
$\D_\theta(\nu)=[0,1]^k$, for $\theta>1$, by means of Lemma \ref{circle}.

In conclusion, the Deffuant model with this initial distribution features the critical value
$$\theta_\text{\upshape c}=\begin{cases}
   1,&\text{for }k=1\text{ or }k=2,3 \text{ and }p\in[1-\tfrac{1}{\sqrt{k}},\tfrac{1}{\sqrt{k}}]\\
   \sqrt{k}\,\max\{p,1-p\},&\text{for }k\geq4 \text{ or } k=2,3 \text{ and }
   p\notin[1-\tfrac{1}{\sqrt{k}},\tfrac{1}{\sqrt{k}}].
   \end{cases}$$
As stated in the above remark, the independence of the individual coordinates is not essential, as long as 
the support stays unchanged. A relation like $\eta_0^{(1)}=1-\eta_0^{(2)}$ in the Bernoulli example with parameter
$p=\tfrac12$ however, will influence both $\supp(\eta_0)$ and as a consequence $\theta_\text{\upshape c}$ as well.
\end{enumerate}
\end{example}

\begin{example}\label{mu}
There is one more crucial change when the opinions in the Deffuant model on $\Z$ are given by vectors instead of
real numbers. The parameter $\mu$, shaping the size of compromising steps, which was of no particular
interest so far, can actually play a crucial role in the critical case.

In order to verify this claim, let us consider the two-dimensional initial distribution given by
$\text{\upshape unif}(\{(0,0),(1,0),(\tfrac{1}{\pi},1)\})$, which is depicted below.
Given $\theta=1$ we have
$$[0,1]\times\{0\}\subseteq\supp_\theta(\eta_t)\text{ for all }t>0,$$
following the reasoning of Example \ref{jump}.
But the point $(\tfrac{1}{\pi},0)$ can only be approximated, never attained by $\eta_t(v)$, if $\mu$
is rational for example. For $\mu=\tfrac{1}{\pi}$ on the other hand, $\eta_t(v)=(\tfrac{1}{\pi},0)$
with positive probability which leads to $\supp(\eta_t)=\conv(\supp(\eta_0))$.

Note that for this distribution, we have $h=1>R$, since $\E\eta_0=\tfrac13\,(1+\tfrac{1}{\pi},1)$.
\par\begingroup \rightskip13.5em\noindent
Similarly to the proof of the above theorem, we can
conclude that the Deffuant model on $\Z$ with confidence bound $\theta=\theta_\text{\upshape c}=1$ and this initial
distribution approaches almost surely no consensus for $\mu\in(0,\tfrac12]\cap\mathbb{Q}$ and 
almost surely strong consensus for $\mu=\tfrac{1}{\pi}$:

If $\mu$ is rational, vertices holding the initial opinion $(\tfrac{1}{\pi},1)$ can never compromise with
such holding an opinion $(a,0)$ since $a$ is rational and can therefore not be $\tfrac{1}{\pi}$. Consequently,
we will have a.s.\ no consensus due to blocked edges.
\par\endgroup
\vspace*{-4.8cm}
 \begin{figure}[H]
     \hspace*{8.5cm}
     \unitlength=0.80mm
     \begin{picture}(45.00,45.00)
          
          \put(-2.00,0.00){\vector(1,0){40.00}}
          \put(10.19,-1.50){\line(0,1){3.00}}
          \put(-1.50,32.00){\line(1,0){3.00}}
          \put(0.00,-2.00){\vector(0,1){40.00}}
          \put(32.00,-1.50){\line(0,1){3.00}}
          \put(-4.30,-5.00){$\scriptstyle 0$}
          \put(-4.30,31.00){$\scriptstyle 1$}
          \put(31.25,-5.00){$\scriptstyle 1$}
          \put(8.70,-6.50){$\scriptstyle \tfrac{1}{\pi}$}
          
          {\color[rgb]{0,0,1}
          \put(7.00,22.00){$\scriptstyle \mathcal{L}(\eta_0)$}
          \put(-0.10,0.00){\circle*{1.5}}
          \put(31.97,0.00){\circle*{1.5}}
          \put(10.14,32.00){\circle*{1.5}}}
          \put(14.03,10.63){\circle*{1.5}}
          \put(16.00,10.00){$\scriptstyle \E\eta_0$}
      \end{picture}
  \end{figure}
\vspace*{0.4cm}
  
If $\mu=\tfrac{1}{\pi}$ however, we can come up with a finite configuration allowing for the local modification,
which guaratees the existence of two-sidedly $\epsilon$-flat vertices. Actually $n=3$ is enough and $$x_1=(1,0),\ 
x_2=(0,0),\ x_3=(\tfrac{1}{\pi},1)$$ will be an appropriate choice of starting values, if the edge sequence
$(e_i)_{i=1}^N$ begins with $e_1=\langle1,2\rangle,\ e_2=\langle2,3\rangle$, since that will bring the value at site
$1$ to $(1-\tfrac{1}{\pi},0)$, the one at $2$ to $(\tfrac{1}{\pi},\tfrac{1}{\pi})$ and the one at
$3$ to $(\tfrac{1}{\pi},1-\tfrac{1}{\pi})$, all lying in $B[\E\eta_0,\tfrac12]$, and thus their pairwise
distances are all less than the confidence bound. If the edge sequence contains the edge pair 
$(\langle1,2\rangle,\langle2,3\rangle)$ enough times, the final values of the finite configuration will all lie
at Euclidean distance at most $\tfrac\epsilon3$ from the initial average $\tfrac13(x_1+x_2+x_3)=\E\eta_0$ for any fixed
$\epsilon>0$. Note that in the present case, when transforming the finite configuration into a part of the dynamics on
the whole line graph, we don't have to worry about taking small balls around the initial values $x_i$ in order to
get an event $C$ with positive probability, since the $x_i$ are atoms of the initial distribution. Taking small balls
would actually invalidate the argument due to the fact that the parameter $\theta$ is pinned to the critical value
$\theta_\text{\upshape c}=1$ not allowing for small marginals.\\[0.5em]
\noindent
Another fact that can be seen from this example is that the jumps of the mapping $\vartheta\mapsto \D_\vartheta(\nu)$
do not have to be continuous from the right in the sense that $\D_\theta(\nu)=\bigcap_{\vartheta>\theta}
\D_\vartheta(\nu)$. Given $\mu\in\mathbb{Q}$ we get for this initial distribution
 $$\D_\theta(\nu)=\begin{cases}
   \supp(\eta_0),&\text{for }\theta<1\\
   [0,1]\times\{0\}\cup\{(\tfrac{1}{\pi},0)\},&\text{for }\theta=1\\
   \conv(\supp(\eta_0)),&\text{for }\theta>1,
   \end{cases}$$
hence there can actually be a double jump.
\end{example}

\section{Metrics other than the Euclidean distance}

Having investigated the changes that multidimensional opinion values cause in the Deffuant model,
another interesting aspect is the impact of the measure of distance between two opinions.
What happens if we apply some general metric $\rho$ other than the natural choice given by the
Euclidean norm?

Although this generalization does not entirely fit the framework as laid out in Section 1, it is not
worth repeating all the definitions as one would simply have to replace all appearing distances $\n{x-y}$
by $\rho(x,y)$ correspondingly. Note however that switching to a general metric $\rho$ influences the
dynamics of the Deffuant model only in determining which opinion values are within `speaking distance',
that is allowing for an update if neighbors with corresponding opinions interact. Once the two
values are close enough in this respect, the updated opinion values will just be the convex combinations
described in (\ref{dynamics}), even if the straight line connecting both values might no longer be the
geodesic between them (as in the Euclidean case) and the steps taken towards the arithmetic average
can be of different length if $\rho$ is not translation invariant.

With respect to the considerations in the foregoing section, the following
properties of a distance measure play an important role.

\begin{definition}\label{extracon}
Consider a metric $\rho$ on $\R^k$. 
\begin{enumerate}[(i)]
\item
Let the metric $\rho$ be called {\itshape sensitive to 
coordinate} $i$, if there exists a function $\varphi:[0,\infty)\to[0,\infty)$ such that 
$\lim_{s\to\infty}\varphi(s)=\infty$ and for any two vectors $x, y\in\R^k$ with $|x_i-y_i|> s$, it holds that
$\rho(x,y)>\varphi(s)$.
\item
Call $\rho$ {\em locally dominated by the Euclidean distance}, if there exist some $\gamma,c>0$
such that for $x,y\in\R^k$ with $\n{x-y}\leq\gamma$ it holds that 
\begin{equation}\label{domi}\rho(x,y)\leq c\cdot||x-y||_2.\end{equation}
\item
Finally, let $\rho$ be called {\itshape weakly convex} if for all $x,y,z\in\R^k$:
$$\rho(x, \alpha y+(1-\alpha)\,z)\leq \max\{\rho(x,y),\rho(x,z)\}\quad\text{for all }\alpha\in[0,1].$$
\end{enumerate}
\end{definition}

\noindent
The convexity of balls $B_\rho(x,r)=\{y\in\R^k,\;\rho(x,y)<r\}$ generated by the metric is a crucial feature. 
It is not hard to check that the balls generated by $\rho$ are convex if and only if the metric is weakly
convex: Sufficiency is obvious, since $y,z\in B_\rho(x,r)$ immediately gives $\conv(\{y,z\})\subseteq B_\rho(x,r)$.
As to necessity, if there are $x,y,z\in\R^k$, $\alpha\in(0,1)$ s.t.\ 
$\rho(x, \alpha y+(1-\alpha)\,z)>\max\{\rho(x,y),\rho(x,z)\} $, we can choose 
$r\in(\max\{\rho(x,y),\rho(x,z)\},\rho(x, \alpha y+(1-\alpha)\,z))$ and conclude that $B_\rho(x,r)$
can not be convex. It should be mentioned that when talking about the metric space $(\R^k,\rho)$, we will
always assume that it is equipped with the Borel $\sigma$-algebra generated by the metric $\rho$.

If $\rho$ is locally dominated by the Euclidean distance, we can find a constant $C=C(\theta)$ such that
(\ref{domi}) holds in fact for all $x,y\in\R^k$ with $\rho(x,y)\leq\theta$ if $c$ is replaced by $C$:
If $\n{x-y}>\gamma$ but $\rho(x,y)\leq\theta$, we can conclude that
$$\rho(x,y)\leq\theta\leq\tfrac{\theta}{\gamma}\,\n{x-y},$$
hence $C:=\max\{c,\tfrac{\theta}{\gamma}\}$ will do.

\begin{definition}
Let the Deffuant model with respect to a general distance measure $\rho$ be defined just as in Section \ref{intro},
with the only change that the restriction of the confidence bound in (\ref{dynamics}) will now rule that Poisson
events cause updates only if $\rho(a,b)\leq\theta$, where $a,b$ denote the opinion values at the corresponding
vertices. As the convexity of balls is enormously important in the analysis presented in the foregoing
section, in what follows $\rho$ will be assumed to be weakly convex.

No consensus still means that we have finally blocked edges, that is some $\langle u,v\rangle$
s.t.\ $\rho(\eta_{t}(u),\eta_{t}(v))>\theta$ for all $t$ large enough. Similarly, the convergence notion in the
definition of consensus is now based on the distance $\rho$.

As before, the initial opinions are i.i.d.\ with some common distribution $\mathcal{L}(\eta_0)$ on $\R^k$. 
If the distribution of $\eta_0$ has a finite expectation, we define its radius with respect to $\rho$ as
$$R_\rho:=\inf\left\{r>0,\;\Prob\big(\eta_0\in B_\rho(\E\eta_0,r)\big)=1\right\},$$
similarly to the Euclidean case, see Definition \ref{radius}.

Likewise, the notion of $\epsilon$-flatness transfers to the new setting as follows:
A vertex $v\in\Z$ is called
{\em $\epsilon$-flat (with respect to $\rho$)} to the right in the initial configuration $\{\eta_0(u)\}_{u\in\Z}$
if for all $n\geq0$:
\begin{equation}\label{rhoflat}
  \frac{1}{n+1}\sum_{u=v}^{v+n}\eta_0(u)\in B_\rho(\E\eta_0,\epsilon),
\end{equation}
similarly for $\epsilon$-flatness to the left and two-sided $\epsilon$-flatness.
\end{definition}

\noindent
By imposing appropriate additional restrictions on the weakly convex metric $\rho$ and the initial distribution,
we can retrieve the result of Theorem \ref{nogap} also in this generalized setting. The extra restriction
on $\mathcal{L}(\eta_0)$ is that $\E[\eta_0^{\,2}]$ is finite, as this is no longer directly implied by the finiteness
of the initial distribution's radius (just think of a bounded metric). The Cauchy-Schwarz inequality implies
that this constraint is equivalent to the finiteness of the entries in the covariance matrix corresponding to
the distribution of $\eta_0$, which is why we will simply refer to it as having a finite second moment, just
as in the univariate case.

Finally, note that if we fix an initial distribution $\mathcal{L}(\eta_0)$, due to the update rule (\ref{dynamics}),
all possible future opinion values lie in the convex hull of its support,
$\conv(\supp\eta_0)$.
For this reason it will suffice in every respect that $\rho$ is weakly convex (and possibly locally dominated by the
Euclidean norm) on $\conv(\supp\eta_0)$ only, not the entire $\R^k$.

\begin{theorem}\label{nogaprho}
In the Deffuant model on $\Z$ with the underlying opinion space $(\R^k,\rho)$ and an initial opinion
distribution $\mathcal{L}(\eta_0)$ we have the following limiting behavior:
\begin{enumerate}[(a)]
\item If $\rho$ is locally dominated by the Euclidean distance and $\mathcal{L}(\eta_0)$ has a finite second 
moment, a finite radius $R_\rho\in[0,\infty)$ and mass around its mean, i.e.\
\begin{equation}\label{rhomatm}
\Prob\big(\eta_0\in B_\rho(\E\eta_0,r)\big)>0 \text{ for all }r>0,
\end{equation}
the critical parameter is $\theta_\text{\upshape c}=R_\rho$, meaning that for $\theta<R_\rho$ we have a.s.\ no
consensus and for $\theta>R_\rho$ a.s.\ strong consensus.
\item Let $\eta_0=(\eta_0^{(1)},\dots,\eta_0^{(k)})$ be the random initial opinion vector. If one of the coordinates
$\eta_0^{(i)}$ has an unbounded marginal distribution (with respect to the absolute value), its expected value
exists (regardless of whether finite, $+\infty$ or $-\infty$) and $\rho$ is sensitive to this coordinate, the limiting
behavior will a.s.\ be no consensus, irrespectively of $\theta$.
\end{enumerate}
\end{theorem}

\begin{proof}
\begin{enumerate}[(a)]
\item The proof of this theorem is exactly the same as the proof of Theorem \ref{nogap}. One only has to check
that the additional requirements on $\rho$ make up for the crucial properties of the Euclidean norm that were
used in the cited proof. The (multivariate) SLLN states that the averages in (\ref{rhoflat}) for large $n$ are
close to the mean in Euclidean distance, hence with respect to $\rho$ due to (\ref{domi}). Local modification
of the initial profile will then guarantee the existence of one-sidedly $\epsilon$-flat vertices.

The crucial role of $\epsilon$-flat vertices is preserved by the weak convexity of $\rho$: The proof of
Prop.\ 5.1 in \cite{ShareDrink} shows that given an edge $\langle v-1,v\rangle$ along which there have been no
updates yet, the opinion value at $v$ is a convex combination of averages as in (\ref{rhoflat}), hence lies
in $B_\rho(\E\eta_0,\epsilon)$ as well, if $v$ was $\epsilon$-flat to the right with respect to the initial
configuration, due to convexity of the $\rho$-balls.

As to the supercritical regime, the a.s.\ existence of two-sidedly $\epsilon$-flat vertices follows from the a.s.\
existence of one-sidedly $\epsilon$-flat vertices and the i.i.d.\ property of the initial configuration, just
as in the Euclidean case. The weak convexity of $\rho$ is needed once more to conclude that the opinion values
of two-sidedly $\epsilon$-flat vertices stay close to the mean, just as in La.\ 6.3 in \cite{ShareDrink}.

When we want to apply the argument of Prop.\ 6.1 in \cite{ShareDrink}, stating that neighbors will a.s.\ either finally
concur or the edge between them be blocked for large $t$, it is essential that condition (\ref{domi}), together
with the finite second moment, allows once again to borrow the energy idea. The extra condition of a finite
second moment implies the finiteness of the expected initial engergy $\E[W_0(v)]=\E[\eta_0(v)^{2}]$, as mentioned just
before the theorem.
If the opinions $\eta_{t}(u),\eta_{t}(v)$ of two neighbors are within the confidence bound with respect to $\rho$ but
$\rho(\eta_{t}(u),\eta_{t}(v))\geq\delta$ for some $\delta>0$, then due to (\ref{domi}):
$||\eta_{t}(u)-\eta_{t}(v)||_2\geq\tfrac{\delta}{C}$, where $C=\max\{c,\tfrac{\theta}{\gamma}\}>0$, see the
comments after Definition \ref{extracon}. This will cause an energy loss of at least $2\mu(1-\mu)(\tfrac{\delta}{C})^2$
when they compromise. Again, this cannot happen infinitely often with positive probability as the expected energy
at time $t=0$ is finite and the expected total energy preserved over time.

\item Given $\rho$ is sensitive to coordinate $i$, the idea of proof of the second claim can be reutilized as well.
The sensitivity leads to the fact that there is some $s>0$ s.t.\ $|x_i-y_i|> s$ implies $\rho(x,y)>\theta$. 
As alluded in the proof of Theorem \ref{nogap}, the arguments used for unbounded distributions in Thm.\ 2.2 in
\cite{Deffuant} show that under the given conditions, there are a.s.\ vertices that differ more than $s$ from both
their neighbors in the $i$th coordinate (with respect to the absolut value) in the initial configuration and this
will not change no matter whom their neighbors will compromise with. Consequently the corresponding opinion vectors
will always be at $\rho$-distance more than $\theta$.
\end{enumerate}
\end{proof}\vspace*{-1em}

 \begin{example}
 \begin{enumerate}[(a)]
 \item The $L^p$-norm for general $p\in[1,\infty]$ on $\R^k$ is defined as follows:
    $$\lVert x\rVert_p:=\Big(\sum_{i=1}^k |x_i|^p\Big)^{\tfrac1p}\quad\text{for } p\in[1,\infty)\text{ and}\quad
    \lVert x\rVert_{\infty}:=\max_{1\leq i\leq k} |x_i|.$$
    In fact, these norms are all equivalent. More precisely, for $1\leq q<p\leq\infty$:
    $$\lVert x\rVert_p\leq\lVert x\rVert_q\leq k^{\big(\tfrac1q-\tfrac1p\big)}\,\lVert x\rVert_p.$$
    This implies for all $p\in[1,\infty]$:
    $$\lVert x\rVert_p\leq\sqrt{k}\;\n{x}.$$
    In other words all induced metrics $\rho(x,y)=\lVert x-y\rVert_p$, are -- to be precise globally -- 
    dominated by the Euclidean distance.
    
    It is easy to check that the norm axioms guarantee the convexity of balls, hence
    the metric induced by $\lVert\,.\,\rVert_p$ is weakly convex for any $p\in[1,\infty]$.
    
    Furthermore, $\lVert x\rVert_p\geq k^{\big(\tfrac1p-1\big)}\lVert x\rVert_1\geq k^{\big(\tfrac1p-1\big)}|x_i|$
    for all $1\leq i\leq k$ implies sensitivity to every coordinate. In conclusion, both parts of Theorem \ref{nogaprho}
    can be applied to the Deffuant model with the metric induced by some $L^p$-norm, i.e.\
    $\rho(x,y)=\lVert x-y\rVert_p$, $p\in[1,\infty]$, as distance measure.
 
 \item    
    If the definition of $\lVert\,.\,\rVert_p$ is extended to values for $p$ in $(0,1)$, the corresponding
    functions are not subadditive, hence do not induce a metric.
    
    Raised to the power $p$, we get the distance measures
    $$\rho_p(x,y):=\big(\lVert x-y\rVert_p\big)^{p}=\sum_{i=1}^k |x_i-y_i|^p,$$ which
    are in fact metrics for all $p\in(0,\infty)$ and obviously sensitive to every coordinate. For $p\in(0,1)$
    these metrics fail to have convex balls. For $p\in[1,\infty)$ however, they are weakly convex which can be seen 
    from the weak convexity of $\lVert\,.\,\rVert_p$ as follows:
    \begin{align*}\rho_p(x, \alpha y+(1-\alpha)\,z)&=\big(\big\lVert x-\big(\alpha y+(1-\alpha)\,z\big)\big\rVert_p\big)^p\\
                                                &\leq\big(\max\{\lVert x-y\rVert_p,\lVert x-z\rVert_p\}\big)^p\\
                                                &=\max\{\rho_p(x,y),\rho_p(x,z)\}.
    \end{align*}
    
    The metrics $\rho_p,\ p\in[1,\infty)$ are no longer equivalent to the Euclidean distance, but still locally
    dominated in the sense of (\ref{domi}). In conclusion, Theorem \ref{nogaprho} equally applies to the Deffuant model
    where distances are taken with respect to $\rho_p$.
    
    More generally, given $\phi=(\phi_i)_{i=1}^k$ with non-negative functions $\phi_i$ defined on $\R_{\geq0}$ we
    can consider
    $$\rho_\phi(x,y):=\sum_{i=1}^k\phi_i\big(|x_i-y_i|\big).$$
    For this to be a proper metric, the $\phi_i$ have to be convex satisfying $\phi_i(s)=0$ if and only if $s=0$.
    Defined this way $\rho_\phi$ is convex, in particular weakly convex. It will be locally dominated by the Euclidean 
    distance by default and sensitive to coordinate $i$ if and only if $\phi_i(s)$ is unbounded as $s\to\infty$.
 \end{enumerate}
 \end{example}  
 
 \begin{example} \label{discr}
  The extra condition (\ref{domi}) cannot be dropped. Let us consider the discrete metric 
  $\rho(x,y)=\mathbbm{1}_{\{x\neq y\}}$ -- which is weakly convex -- on $\R$. Clearly, it is not locally dominated by the
  Euclidean metric. Let $\eta_0$ have the mixed distribution with constant density $\tfrac14$ on $[-1,1]$ and point
  mass $\tfrac12$ at $0$. Hence  $\mathcal{L}(\eta_0)$ has expectation $0$ and radius 1 (actually both with respect to
  $\rho$ and the Euclidean distance). Regarding (\ref{rhomatm}), we find $\Prob(\eta_0\in B_\rho(0,\epsilon))\geq\tfrac12$
  for all $\epsilon\geq0$. Take $\mu\in(0,\tfrac12]$ to be a transcendental number (e.g.\ $\tfrac{1}{\pi}$).
  Furthermore, we choose $\theta\geq2$ which obviously makes blocked edges impossible.
  
  At every time $t$, $\eta_t(v)$ is a finite (but random) convex combination of the initial opinions 
  $\{\eta_0(y)\}_{y\in\Z}$, say
  \begin{equation}\label{SADrep}\eta_t(v)=\sum_{y\in\Z}\xi_{v,t}(y)\,\eta_0(y),\end{equation}
  which is the SAD representation, see La.\ 3.1 in \cite{ShareDrink}. Almost surely, there are two edges that do not
  experience Poisson events up to time $t$ and enclose $v$. It is not hard to show -- by induction on the (a.s.\ finitely
  many) Poisson events occurring up to time $t$ on the edges between those two -- that the non-zero factors $\xi_{v,t}(y)$
  in the representation of $\eta_t(v)$ are (random) polynomials in $\mu$ with integer coefficients. Furthermore, for
  $y\neq v$ they have no constant term, for $y=v$ the constant term equals $1$:
  At time $0$ we find $\xi_{u,0}(y)=\mathbbm{1}_{\{u=y\}}$ for all $u,y\in\Z$. With a Poisson event at
  time $s$ on the edge $\langle u, u+1\rangle$ that actually causes an update, the coefficients change according to
  \begin{equation*}\begin{array}{rl}\xi_{u,s}(y)&\!\!\!=\,(1-\mu)\,\xi_{u,s-}(y)+\mu\,\xi_{u+1,s-}(y)\\
                 \xi_{u+1,s}(y)&\!\!\!=\,\mu\,\xi_{u,s-}(y)+(1-\mu)\,\xi_{u+1,s}(y),\end{array}
   \end{equation*}
  for all $y\in\Z$, compare with (\ref{transf}). This establishes the induction step.\\[1em]
  \noindent
  Using the representation (\ref{SADrep}) we find for two neighbors $u,v$: 
  $$\eta_t(v)-\eta_t(u)=\sum_{y\in\Z}\big(\xi_{v,t}(y)-\xi_{u,t}(y)\big)\,\eta_0(y).$$
  As $\xi_{v,t}(v)-\xi_{u,t}(v)$ is a non-zero polynomial in $\mu$ with integer coefficients, it cannot be zero.
  Additionally, due to the fact that $\theta\geq2$, the $\xi$-factors only depend on the Poisson events, which implies
  that the two random variables 
  $$X:=\frac{1}{\xi_{v,t}(v)-\xi_{u,t}(v)}\;\sum_{y\neq v}\big(\xi_{v,t}(y)-\xi_{u,t}(y)\big)\,\eta_0(y)$$
  and $\eta_0(v)$ are independent. Since $\Prob(\eta_0(v)=0)=\Prob(\eta_0(v)\neq0)=\tfrac12$, we get
  $$\Prob(\eta_t(v)-\eta_t(u)\neq0)\geq\Prob(X=0,\eta_0(v)\neq0)+\Prob(X\neq0,\eta_0(v)=0)=\tfrac12.$$
  This leads to
  $$\Prob\Big(\limsup_{t\to\infty}\rho\big(\eta_t(u),\eta_t(v)\big)=1\Big)\geq\tfrac12$$
  for all neigbors $u,v$, which renders even weak consensus impossible.
  
  In fact, with this choice of initial distribution and metric, the Deffuant model exhibits a limiting behavior
  that is not a.s.\ approaching one of the scenarios described in Definition \ref{states}, since it does not
  feature blocked edges, nor almost sure consensus formation in the long run -- instead at any time $t$ the opinions
  of two neighbors are with probability at least $\tfrac12$ at distance $1$, always at speaking terms but not converging.

  Since the choice of $\theta$ is trivial, we can find out what happens by looking at the Deffuant model employing
  the Euclidean distance instead. By Theorem \ref{nogap} all opinions will a.s.\ approach the mean $0$,
  but whenever two of them do not coincide they are at $\rho$-distance 1.  
\end{example}  

\begin{example} 
  To illustrate the importance of the sensitivity in part (b) of Theorem \ref{nogaprho}, let us consider the
  two metrics $d(x,y)=\n{x-y}$, that is the Euclidean metric, and 
   $$\rho(x,y)=\begin{cases}\n{x-y},&\text{if }\n{x-y}\leq1\\
                            1,&\text{otherwise.}
   \end{cases}$$
  Evidently, $\rho$ is not sensitive to any coordinate and that it is weakly convex is not hard to check either:
  For $r<1$ the balls $B_{\rho}(x,r)$ are the same as the Euclidean balls, for $r\geq1$ we get
  $B_{\rho}(x,r)=\R^k$. So in either case it is a convex set.
  
  For simplicity, let us take $k$ to be $1$ -- the Euclidean distance is then induced by the absolute value -- 
  and choose the standard normal distribution $\mathcal{N}(0,1)$ as initial distribution. Due to $\rho(x,y)\leq|x-y|$, 
  $\rho$ is locally dominated by the Euclidean distance. As the normal distribution has a finite second moment
  and mass around its mean, part (a) of Theorem \ref{nogaprho} shows that in the Deffuant model using $\rho$
  as the distance measure, the radius $R_\rho=1$ marks the critical value for $\theta$ at which we have a
  phase transition from a.s.\ no consensus to a.s.\ strong consensus.
  
  In the Deffuant model using the Euclidean distance however, there will a.s.\ be no consensus irrespectively
  of $\theta$ according to Theorem \ref{gen} (b). 
\end{example} 
\vspace{0.5em}
\noindent
The final aim will now be to prove a generalization of Theorem \ref{gapsEucl} to the Deffuant model with general
metric $\rho$ instead of the Euclidean. In order to be able to do this we have to transfer the necessary auxiliary
results leading to Theorem \ref{gapsEucl}, essentially by replacing all occurring 
Euclidean distances by distances with respect to $\rho$, however it requires small adjustments.

\begin{definition}
Consider a random variable $\xi$ on $(\R^k,\rho)$. The {\itshape support} of its distribution is the following
subset of $\R^k$, closed with respect to $\rho$:
\begin{equation}\label{suprho}
\supp(\xi):=\left\{x\in\R^k,\;\Prob\big(\xi\in B_\rho(x,r)\big)>0\ \text{for all }r>0\right\}.
\end{equation}
\end{definition}

\begin{remark}
The last argument in the proof of Proposition \ref{Radius} shows $\supp(\eta_0)\subseteq B_\rho[\E\eta_0,R_\rho]$
for all initial distributions bounded with respect to $\rho$. The first part of its proof, i.e.\ showing that
$\supp(\eta_0)\subseteq B[\E\eta_0,r]$ implies $\Prob(\eta_0\in B[\E\eta_0,r])=1$, is based on the
theorem of Heine-Borel, stating that closed and bounded sets are compact in $(\R^k,\n{\,.\,})$, which does not
hold for general metric spaces.
For the discrete metric (see Example \ref{discr}) and a probability measure without point masses, the set defined
in (\ref{suprho}) is in fact empty. 

If however $(\R^k,\rho)$ is separable -- i.e.\ there exists a countable dense subset -- we get 
$\Prob(\xi\in\supp(\xi))=1$ for any random variable $\xi$ (see e.g.\ Thm.\ 2.1, p.\ 27 in \cite{Partha}),
and thus the full statement of Proposition \ref{Radius}.

Given $\rho$ is locally dominated by the Euclidean distance, we can immediately conclude that $(\R^k,\rho)$ is separable,
since due to (\ref{domi}) the set $\mathbb{Q}^k$ is not only dense in $(\R^k,\n{\,.\,})$ but also in $(\R^k,\rho)$.

In conclusion, if $(\R^k,\rho)$ is separable and $\eta_0$ has a finite expectation, its distribution's radius can be
written as $R_\rho=\sup\{\rho(\E\eta_0,x),\; x\in\supp(\eta_0)\}$.
\end{remark}

\noindent
Adjusting the definition of $\D_\theta(\nu)$ (see Definition \ref{Dtheta}) to the general setting by substituting
$\rho$-balls for Euclidean balls -- let us denote the resulting set by $\D_\theta^\rho(\nu)$ -- allows to reuse the
arguments in the lemmas dealing with its properties.
Although referencing to Proposition \ref{Radius}, in order to prove Lemma \ref{properties} only 
$\supp(\eta_0)\subseteq B[\E\eta_0,R]$ was needed, hence its statement is true for any weakly convex $\rho$ --
with the terms related to closure now referring to the topology generated by $\rho$.

As the final conclusions similar to Theorem \ref{gapsEucl} will require $\rho$ to be locally dominated by the
Euclidean distance, let us assume for the remainder of this section that $\rho$ is not only weakly convex but
also (\ref{domi}) holds.

When it comes to the central Lemma \ref{circle}, the claims that can be modified to hold for such $\rho$ 
as well without major efforts read as follows (again connectedness and closure refer to the topology generated by $\rho$):

\begin{lemma}\label{circlerho}
Let $\rho$ be a weakly convex metric locally dominated by the Euclidean distance.
\begin{enumerate}[(a)]
  \item For all $x\in\R^k$ and $0\leq \delta<\tfrac{\theta}{2}$, the set $\D_\theta^\rho(\nu)\cap B_\rho[x,\delta]$
        is convex.
  \item The connected components of $\D_\theta^\rho(\nu)$ are convex and at $\rho$-distance at least
        $\theta$ from one another. If $\D_\theta^\rho(\nu)$ is connected, then
        $\D_\theta^\rho(\nu)=\overline{\conv(\supp(\eta_0))}$.
  \item If $R_\rho<\infty$ and $\nu$ has mass around its mean, i.e.\ condition (\ref{rhomatm}) holds, then
        $\D_{\theta}^\rho(\nu)=\overline{\conv(\supp(\eta_0))}$ for all $\theta>R_\rho$.
  \item If $\D_\theta^\rho(\nu)$ is connected and $\E\eta_0$ finite, then $\E\eta_0\in\D_\theta^\rho(\nu)$
\end{enumerate}
\end{lemma}

\begin{proof}
The proof is essentially identical to the one of Lemma \ref{circle}. In part (a) we only have to choose
$m,n\in\N$ such that
$$\left|\frac{m}{m+n}-\alpha\right|\leq
\frac{\min\{\tfrac{r}{4c},\tfrac{\gamma}{2}\}}{\max\{\n{y},\n{z}\}}.$$
Then 
$$\n{(\tfrac{m}{m+n}\,y+\tfrac{n}{m+n}\,z)-(\alpha y+(1-\alpha)z)}\leq
|\tfrac{m}{m+n}-\alpha|\cdot\n{y}+|\alpha-\tfrac{m}{m+n}|\cdot\n{z}\leq\gamma,$$
which together with (\ref{domi}) implies
 \begin{align*}
         \rho\big(\eta_N(v),\alpha y+(1-\alpha)z\big)&\leq
         \tfrac{r}{2}+\rho\big(\tfrac{m}{m+n}\,y+\tfrac{n}{m+n}\,z,\alpha y+(1-\alpha)z\big)\\
         &\leq\tfrac{r}{2}+c\,\n{(\tfrac{m}{m+n}\,y+\tfrac{n}{m+n}\,z)-(\alpha y+(1-\alpha)z)}\\
         &\leq \tfrac{r}{2}+c\,(|\tfrac{m}{m+n}-\alpha|\cdot\n{y}+|\alpha-\tfrac{m}{m+n}|\cdot\n{z})
         \leq r.
 \end{align*}

\vspace{0.5em}
\noindent
As to part (b), we can follow the first part of the proof of Lemma \ref{circle} (c) replacing every
Euclidean distance by $\rho$ until the angles are considered. Since $B_\rho[x_j,\tfrac{\theta}{2}]$
might be oddly shaped, we can define $r:=\min\{\tfrac{\theta}{2c},\gamma\}>0$ and consider the Euclidean
ball $B[x_j,r]$ which by (\ref{domi}) is contained in $B_\rho[x_j,\tfrac{\theta}{2}]$. Cutting short
an angle $\alpha$ as described there, will now reduce the (Euclidean) length of the polygonal chain
by at least $2r\cdot(1-\cos(\alpha))$ and the argument goes through yielding that the Euclidean
closure of the component $C$ connected with respect to $\rho$ contains $\conv(\{x,y\})$. It follows from
the generalized statement of Lemma \ref{properties} that being a component of $\D_\theta^\rho(\nu)$, $C$
is $\rho$-closed. This in turn implies that $C$ is also closed with respect to the Euclidean
distance, using (\ref{domi}), and hence containing $\conv(\{x,y\})$. The rest of the claim easily
follows, again by replacing $\n{x-y}$ by $\rho(x,y)$.

\vspace{1em}
\noindent
Part (c) is an easy consequence of the arguments leading to (a) and (b) that can be verified just
as in the proof of Lemma \ref{circle} (d).

\vspace{1em}
\noindent
Finally, the only insight needed to accept the proof of Lemma \ref{circle} (f) as proof of claim (d)
above is that $\D_\theta^\rho(\nu)$, being closed in $(\R^k,\rho)$, is also closed in the Euclidean
space $(\R^k,\n{\,.\,})$, due to (\ref{domi}).
\end{proof}

\begin{definition}
Corresponding to Definition \ref{suppdef}, let the support of the distribution of $\eta_t$ in the
Deffuant model with parameter $\theta$ and distance measure $\rho$ be denoted by
$\supp_\theta^\rho(\eta_t)$.

Respectively, the length of the largest gap in $\supp(\eta_0)$ with respect to $\rho$ will be given by 
$$h_\rho:=\inf\{\theta>0,\;\D_\theta^\rho(\nu)\text{ is connected in } (\R^k,\rho)\},$$
compare with Definition \ref{gap}.
\end{definition}

\noindent
Following the arguments in the proof of Lemma \ref{suppt} with scrutiny reveals that the corresponding
statements are also true for $\supp_\theta^\rho(\eta_t)$ in place of $\supp_\theta(\eta_t)$ and
$B_\rho[\E\eta_0,R_\rho]$ substituting $B[\E\eta_0,R]$ -- actually even for metrics which are only weakly
convex and not locally dominated by the Euclidean distance for only the convexity of
$B_\rho[\E\eta_0,R_\rho]$ is needed. Concerning Proposition \ref{supp_t} however,
we will not bother with the proof of a similar statement for the Deffuant model with general $\rho$.
The only fact needed in the upcoming theorem is 
$$\supp_\theta^\rho(\eta_t)\subseteq\D_{\theta+\epsilon}^\rho(\nu)\quad\text{for }\epsilon>0,$$
which readily follows from the last argument in the proof of this very proposition. Having followed up
the crucial intermediate steps makes it possible to slightly modify the proof of Theorem \ref{gapsEucl}
in order to get an argument establishing the following result:

\begin{theorem}\label{gapsrho}
Consider the Deffuant model on $\Z$ with opinion values in $(\R^k,\rho)$, where the corresponding
distance measure $\rho$ is a weakly convex metric, locally dominated by the Euclidean distance.
Assume it features an initial opinion distribution which has a finite second moment and is bounded with
respect to $\rho$, i.e.\
$$R_\rho=\inf\left\{r>0,\;\Prob\big(\eta_0\in B_\rho[\E\eta_0,r]\big)=1\right\}<\infty.$$
If $h_\rho$ denotes the length of the largest gap in its support, then
the critical value for the confidence bound, where a phase transition from a.s.\ no consensus
to a.s.\ strong consensus takes place is $\theta_\text{\upshape c}=\max\{R_\rho,h_\rho\}$.
\end{theorem}

\begin{proof}
As mentioned, the reasoning follows closely the proof of Theorem \ref{gapsEucl}. In case (i), where
$\theta<h_\rho$ we can conclude from Lemma \ref{suppt} and the above remarks that for $\epsilon>0$
such that $\theta+\epsilon<h_\rho$ it follows that
$$\supp(\eta_0)\subseteq\supp_\theta^\rho(\eta_t)\subseteq\D_{\theta+\epsilon}^\rho(\nu).$$
The set $\D_{\theta+\epsilon}^\rho(\nu)$ is not connected (with respect to $\rho$) by definition of
$h_\rho$, hence comprises convex components $C_1$ and $C_2$ at $\rho$-distance at least $\theta+\epsilon$ (see
Lemma \ref{circlerho}). Again, we can choose the components such that $\Prob(\eta_0\in C_i)>0$ for $i=1,2$,
since if we had $\Prob(\eta_0\in C_1)=1$, the fact that $C_1$ is closed with respect to $\rho$
would give $\supp(\eta_0)\subseteq C_1$ and so (using its convexity and the generalization of Lemma \ref{properties})
$$\D_{\theta+\epsilon}^\rho(\nu)\subseteq\overline{\conv(\supp(\eta_0))}\subseteq C_1.$$
But $C_1=\D_{\theta+\epsilon}^\rho(\nu)$ contradicts the disconnectedness.

Consequently, for a fixed vertex $v$ independence of the initial opinions guarantees that the event
$\{\eta_0(v)\in C_1,\eta_0(v+1)\in C_2\}$ has positive probability, which dooms the edge $\langle v, v+1\rangle$
to be blocked by $\rho(\eta_t(v),\eta_t(v+1))\geq\theta+\epsilon$ for all $t\geq0$. Indeed, in the Deffuant model
with parameter $\theta$, $\eta_t(v)$ can not leave the convex set $C_1$ since 
$\supp_\theta^\rho(\eta_t)\setminus C_1$, being a subset of $\D_{\theta+\epsilon}^\rho(\nu)\setminus C_1$,
is at distance at least $\theta+\epsilon$ to $C_1$ for all $t$. The same holds for $\eta_t(v+1)$ and $C_2$
respectively. Due to ergodicity, the existence of blocked edges is therefore an almost sure event.\\[1em]
\noindent
The analysis of case (ii), $\theta<R_\rho$, requires likewise only minor adjustments of the argument in
the proof of Theorem \ref{gapsEucl}. To begin with, the finite second moment of $\eta_0$ implies $\E\eta_0\in\R^k$,
which is not ensured by $R_\rho<\infty$ itself. Let this time $y$ be an element of $\supp(\eta_0)\setminus
B_\rho[\E\eta_0,\theta+2\epsilon]$, which is non-empty for $\epsilon\in(0,\tfrac{R-\theta}{2})$.
Since both $B_\rho[y,\theta+\epsilon]$ and $B_\rho[\E\eta_0,\epsilon]$ are
convex and closed -- with respect to $\rho$ and thus $\n{\,.\,}$ due to (\ref{domi}) -- as well as disjoint, we can
choose $z_1\in B_\rho[y,\theta+\epsilon]$ and $z_2\in B_\rho[\E\eta_0,\epsilon]$ such that 
$$\n{z_1-z_2}=\min\{\n{a-b},\;a\in B_\rho[y,\theta+\epsilon]\text{ and }b\in B_\rho[\E\eta_0,\epsilon]\}>0$$
and then define $z=\tfrac12(z_1+z_2)$ and the half-space $H$ with respect to this point $z$ accordingly.
Note that $H$ contains $B_\rho[\E\eta_0,\epsilon]$ and is disjoint from $B_\rho[y,\theta+\epsilon]$, just as in the 
Euclidean setting, because of the convexity of $\rho$-balls and the choice of $z_1,z_2$.
Moreover, the local domination property (\ref{domi}) forces $B_\rho[\E\eta_0,\epsilon]$ to be a superset of
$B[\E\eta_0,\delta]$, where $\delta=\min\{\tfrac{\epsilon}{c},\gamma\}$, and thus that $\E\eta_0$ lies in the
Euclidean interior of $H$. Having established this, we can follow the rest of the argument (beginning with (\ref{SLLN}),
which again follows from the finite second moment of $\eta_0$) literally, having in mind that $y$ has $\rho$-distance larger than $\theta+\epsilon$ to $H$.\\[1em]
\noindent
Finally, in the supercritical case (iii), i.e.\ $\theta>\max\{R_\rho,h_\rho\}$, we only have to take
Lemma \ref{circlerho} as a replacement for Lemma \ref{circle} and again write $\rho$ for the appearing Euclidean
distances. It is crucial to notice, that limits with respect to the Euclidean distance as in the SLLN and
(\ref{outerpart}) are also limits with respect to $\rho$, once again using (\ref{domi}). Furthermore, in several
places either the triangle inequality or the convexity of Euclidean balls was used, but being a weakly convex
metric, $\rho$ has the corresponding properties. Using the idea of energy to conclude that two neighbors will a.s.\
either finally concur or end up with opinions further than $\theta$ apart from each other, the fact that
$\rho$ is locally dominated by the Euclidean distance is indispensable and employed as in the proof of Theorem
\ref{nogaprho} (a). This is also where the finiteness of the second moment is needed.
\end{proof}

\begin{example}
In order to discern in how far the results of this section do actually add to the univariate case as well,
let us finally consider a metric on $\R$ which is not translation invariant. One can take for example 
$\rho(x,y)=|x^3-y^3|$ for all $x,y\in\R$. This metric $\rho$ obviously generates convex balls, in other
words is weakly convex. However, since 
$$\frac{|x^3-y^3|}{|x-y|}=|x^2+xy+y^2|\to \infty\quad \text{as } x,y\to\infty$$
it is not locally dominated by the absolut value. Nevertheless, as long as we consider a fixed bounded
distribution this problem can be overcome -- as was pointed out just before Theorem \ref{nogaprho} --
since on any bounded interval (\ref{domi}) holds for $\rho$ and some properly chosen $c>0$.

If we consider the initial distribution $\nu=\text{\upshape unif}\{-\tfrac12,\tfrac12\}$, which has 
radius $R_\rho=\tfrac18$, we can conclude from Theorem \ref{gapsrho}, that the critical value for the confidence
bound is $\theta_\text{\upshape c}=\rho(-\tfrac12,\tfrac12)=\tfrac14$. Unlike the Euclidean case, this value will
change with a translation of the initial distribution: Taking $\eta_0+\tfrac32$ instead of $\eta_0$, in other
words $\nu=\text{\upshape unif}\{1,2\}$ as marginal distribution for the initial configuration, we find
$R_\rho=\tfrac{37}{8}$ and $\theta_\text{\upshape c}=\rho(1,2)=7$.
\end{example}

\subsection*{Acknowledgements}
First of all I would like to thank a referee for valuable comments to an earlier draft. Furthermore,
I am very grateful to my supervisor Olle Häggström for helpful discussions of the topic and his constant
support. I would also like to thank Peter Hegarty for bringing up the question about multidimensional opinion
spaces after my talk about the Deffuant model at the Workshop on Discrete Random Geometry in Varberg.


\vspace{0.5cm}
\makebox[0.8\textwidth][l]{
	\begin{minipage}[t]{\textwidth}
	{\sc \small Timo Hirscher\\
   Department of Mathematical Sciences,\\
   Chalmers University of Technology,\\
   412 96 Gothenburg, Sweden.}\\
   hirscher@chalmers.se
	\end{minipage}}

\end{document}